\documentclass[leqno]{amsart}

\usepackage{amsmath, amssymb, amsthm}
   \swapnumbers
   \makeatletter
     \def\swappedhead@plain#1#2#3{%
       \thmnumber{(\textup{#2})}
       \thmname{\@ifnotempty{#2}{~}\textup{#1}}
       \thmnote{ {\textup{(#3)}}}}
     \let\swappedhead\swappedhead@plain
   \makeatother
   \theoremstyle{plain}
   \newtheorem{theo}[equation]{Theorem}
   \newtheorem{prop}[equation]{Proposition}
   \newtheorem{lemm}[equation]{Lemma}
   \newtheorem{coro}[equation]{Corollary}
   \theoremstyle{definition}
   \newtheorem{defi}[equation]{Definition}
   \newtheorem{exem}[equation]{Example}
   \newtheorem{rema}[equation]{Remark}
   \newtheorem{rems}[equation]{Remarks}
   \newtheorem{scho}[equation]{Scholium}

   \numberwithin{equation}{section}
\usepackage[T1]{fontenc}               
\usepackage[utf8]{inputenc}            
   
   \let\dotlessi\i 
   
   \let\polishl\l
   \let\norwegiano\o
   
   \let\russianbreve\u
\usepackage{bm}
\usepackage{epigraph}
\usepackage[dvipsnames]{xcolor}        
\usepackage{tikz-cd}
\usepackage[final,%
            bookmarksnumbered,%
            colorlinks=true,%
            linkcolor=RedViolet,%
            citecolor=RedViolet,
            urlcolor=RoyalBlue]{hyperref}   

  \DeclareMathSymbol{A}{\mathalpha}{operators}{`A}
  \DeclareMathSymbol{B}{\mathalpha}{operators}{`B}
  \DeclareMathSymbol{C}{\mathalpha}{operators}{`C}
  \DeclareMathSymbol{D}{\mathalpha}{operators}{`D}
  \DeclareMathSymbol{E}{\mathalpha}{operators}{`E}
  \DeclareMathSymbol{F}{\mathalpha}{operators}{`F}
  \DeclareMathSymbol{G}{\mathalpha}{operators}{`G}
  \DeclareMathSymbol{H}{\mathalpha}{operators}{`H}
  \DeclareMathSymbol{I}{\mathalpha}{operators}{`I}
  \DeclareMathSymbol{J}{\mathalpha}{operators}{`J}
  \DeclareMathSymbol{K}{\mathalpha}{operators}{`K}
  \DeclareMathSymbol{L}{\mathalpha}{operators}{`L}
  \DeclareMathSymbol{M}{\mathalpha}{operators}{`M}
  \DeclareMathSymbol{N}{\mathalpha}{operators}{`N}
  \DeclareMathSymbol{O}{\mathalpha}{operators}{`O}
  \DeclareMathSymbol{P}{\mathalpha}{operators}{`P}
  \DeclareMathSymbol{Q}{\mathalpha}{operators}{`Q}
  \DeclareMathSymbol{R}{\mathalpha}{operators}{`R}
  \DeclareMathSymbol{S}{\mathalpha}{operators}{`S}
  \DeclareMathSymbol{T}{\mathalpha}{operators}{`T}
  \DeclareMathSymbol{U}{\mathalpha}{operators}{`U}
  \DeclareMathSymbol{V}{\mathalpha}{operators}{`V}
  \DeclareMathSymbol{W}{\mathalpha}{operators}{`W}
  \DeclareMathSymbol{X}{\mathalpha}{operators}{`X}
  \DeclareMathSymbol{Y}{\mathalpha}{operators}{`Y}
  \DeclareMathSymbol{Z}{\mathalpha}{operators}{`Z}

  \newcommand\CC{{\mathbf C}}        
  \newcommand\RR{{\mathbf R}}        
  \newcommand\SO[1]{\operatorname{SO}_{#1}}
  \newcommand\TT{{\mathbf T}}        
  \newcommand\ZZ{{\mathbf Z}}        
  
  \newcommand\Lie{\mathfrak}
  \newcommand\LG{{\Lie{g}}}
  \newcommand\LH{{\Lie{h}}}
  
  \newcommand\LT{{\Lie{t}}}

\let\Im\relax

\DeclareMathOperator\ann{ann}                                
\DeclareMathOperator\Hom{Hom}
\DeclareMathOperator\Im{Im}
\DeclareMathOperator\Ker{Ker}

\DeclareMathOperator\Res{Res}
  \newcommand\e[1]{{\mathrm e^{\hspace{.06em}#1}}}           
\renewcommand\i{{\mathrm i}}                                 
\renewcommand\j{{\mathrm j}}
  \newcommand\IND[3]{\smash{\operatorname{Ind}_{#1}^{#2}#3}}
  \newcommand\inv{^{-1}}                                     
\renewcommand\l[1]{_{\ell,#1}}
  
  \newcommand\subh{_{\,\smash{|}\LH}}

  \newcommand\<{\langle}                                     
\renewcommand\>{\rangle}                                     

\allowdisplaybreaks
\usepackage{bookmark}
\usepackage[mathscr]{eucal}
\usepackage[new]{old-arrows}
\usepackage[neveradjust]{paralist}
\usepackage{array}

\DeclareSymbolFont{lettersA}{U}{txmia}{m}{it}
\SetSymbolFont{lettersA}{bold}{U}{txmia}{bx}{it}
\DeclareFontSubstitution{U}{txmia}{m}{it}
\DeclareMathSymbol{\thetaup}{\mathord}{lettersA}{18}  

\newtheorem{exes}[equation]{Examples}
\newtheorem{s-intertwiner}[equation]{Example: Symplectic intertwiner space {\cite[§6]{Guillemin:1982}}}
\newtheorem{s-induced}[equation]{Example: Induced Hamiltonian G-space \cite{Kazhdan:1978,Weinstein:1978}}
\newtheorem{p-intertwiner}[equation]{Example: Prequantum intertwiner space {\cite[8.1]{Ratiu:2022}}}
\newtheorem{p-induced}[equation]{Example: Induced prequantum G-space {\cite[6.2]{Ratiu:2022}}}
\newtheorem{Peter-Weyl}[equation]{Example: $\mathbf{L^2(G)}$}
\newtheorem{spherical-harmonics}[equation]{Example: $\mathbf{L^2(S^2)}$}
\newtheorem{s-induced-from-0}[equation]{Example: $\mathbf{\textbf{Ind}_H^G\,\{0\}=T^*(G/H)}$}

  \newcommand\form{\Omega}
  \newcommand\level{C}
  \newcommand\subtilde[2]{_{\smash{\tilde{#1}_{#2}}}}
  \newcommand\varpired{\varpi_{\smash{\tilde X}/\!\!/G}}
  \newcommand{\partialflip}[2]{\reflectbox{$#1{#2}$}}
\renewcommand{\partial}{{\mathpalette\partialflip6}}
\renewcommand{\to}{\varto}
\renewcommand{\mapsto}{\varmapsto}
\renewcommand{\hookrightarrow}{\varhookrightarrow}
\renewcommand{\longhookrightarrow}{\varlonghookrightarrow}
\DeclareMathOperator\closure{cl}
\DeclareMathOperator\id{id}
\DeclareMathOperator\interior{int}
\DeclareMathOperator\Liouv{Liouv}

\begin{document}

\title[Diffeological Frobenius Reciprocity]{Remarks on Diffeological Frobenius Reciprocity}

\author{Gabriele Barbieri}
\address{Dipartimento di Matematica e Applicazioni, Universit\`a di Milano-Bicocca, Via Cozzi 55, 20125 Milano, Italy; and Dipartimento di Matematica ``Felice Casorati'', Universit\`a di Pavia, Via Ferrata 5, 27100 Pavia, Italy}
\email{gabriele.barbieri01@universitadipavia.it}

\author{Jordan Watts}
\address{Department of Mathematics, Central Michigan University, Mount Pleasant, MI 48859, USA}
\email{jordan.watts@cmich.edu}

\author{François Ziegler}
\address{Department of Mathematical Sciences, Georgia Southern University, Statesboro, GA 30460-8093, USA}
\email{fziegler@georgiasouthern.edu}

\date{June 2, 2025}

\begin{abstract}
A recent paper \cite{Ratiu:2022} established ``Frobenius reciprocity'' as a bijection $t$ between certain symplectically reduced spaces (which need not be manifolds), and conjectured: 1$^\circ$) $t$ is a diffeomorphism when these spaces are endowed with their natural subquotient diffeologies, 2$^\circ$) $t$ respects the reduced diffeological $2$-forms they may (or might not) carry. In this paper, we prove both this conjecture and a similar one on prequantum reduction, and also give new sufficient conditions for the reduced forms to exist. We stop short of proving that they always exist.
\end{abstract}

\maketitle

\section*{Introduction}

Let $G$ be a Lie group with closed subgroup $H$.  There is a well-known correspondence between the representation theory of $G$ and the Hamiltonian geometry of $G$-spaces.  Constructions such as the induced representation functor $\mathrm{Ind_H^G}$ and the space of intertwiners $\mathrm{Hom_G}$ on the representation theory side have geometric analogues defined using \emph{symplectic reduction}: see \cite{Kazhdan:1978, Weinstein:1978, Guillemin:1982, Guillemin:1983} or (\ref{s-intertwiner},\,\ref{s-induced}) below. Natural properties are expected to hold: thus for instance, given a Hamiltonian $G$-space $X$ and a Hamiltonian $H$-space~$Y$, the paper \cite{Ratiu:2022} establishes ``Frobenius reciprocity'' as a bijection
\begin{equation}
	\label{t}
	t : \Hom_G(X,\IND HGY) \to \Hom_H(\Res^G_HX,Y)
\end{equation}
where $\Res^G_H$ means restriction. But of course one expects more than a \mbox{bijection}: in favorable cases, both sides of \eqref{t} are symplectic manifolds, and so $t$ should be a symplectomorphism. Then again, in general both sides are \emph{singular} reduced spaces, so the structures that $t$ should respect elude traditional differential geometry.

\emph{Diffeology} is a lightweight framework, introduced in the 1980s by Souriau and his students, to make sense of differential geometrical notions on singular (or also infinite-dimensional) spaces in a uniform way, unshackled from the often artificial constraints of topology or functional analysis. Thus for instance, the reduced spaces in \eqref{t} come canonically equipped with (``subquotient'') diffeologies and an attendant Cartan--de Rham calculus of (diffeological) differential forms \cite{Souriau:1985a, Souriau:1988, Iglesias-Zemmour:2013}. This allowed \cite{Ratiu:2022} to conjecture that 1$^\circ$) $t$ is a diffeomorphism, and 2$^\circ$) $t$ respects the reduced 2-forms that the two sides of \eqref{t} may carry.

In Part I of this paper, we prove this conjecture (Theorem \ref{symplectic_frobenius}) as well as a similar one on prequantum $G$-spaces (Theorem \ref{prequantum_frobenius}). We do so \emph{agnostically} as to the existence of reduced forms, i.e.~we show: if one side of \eqref{t} carries a reduced form, then so does the other, and the diffeomorphism $t$ maps one to the other. In a concluding, largely expository section requested by the referee, we describe and exemplify known and further expected relations of these results to the representation-theoretic ancestor (Theorem~\ref{unitary_frobenius}) which motivated them.

In Part II we offer, for the first time in a long time, new sufficient conditions for a reduced space $\Phi\inv(0)/G$ to carry a reduced form; see (\ref{reduced_2-form}, \ref{reduced_1-form}) for precise definitions. To our knowledge, this question was first addressed diffeologically in \cite{Watts:2012, Karshon:2016}, which imply the existence of reduced forms whenever the action of $G$ (or its identity component) on $\Phi\inv(0)$ is locally free \emph{and} proper; see (\ref{two_settings},~\ref{two_prequantum_settings}). Briefly, we show that locally free \emph{or} proper suffices, and we add the third option that the action be \emph{strict}. In more detail:

• Strict actions (Theorem \ref{reduced_forms_for_strict_actions}) are a slight generalization of Iglesias' notion of `principal' action \cite[3.3.1]{Iglesias:1985}, and just what is needed to be able to directly apply Souriau's criterion for integral invariants \cite[2.5c]{Souriau:1985a}.

• For locally free actions (Theorem \ref{reduced_forms_at_regular_values}), we assume $G$ is connected. Our result is then an easy consequence of a theorem on foliations by Hector, Macías-Virgós and Sanmartín-Carbón \cite{Hector:2011}, recently proved anew in \cite[5.10]{Miyamoto:2023}. In the latter paper, Miyamoto makes inroads on a generalization to singular foliations: that may be the most promising route to obtaining reduced forms in greater generality than we were able to.  In the unpublished Appendix of \cite{Lin:2023}, Theorem B.9 suggests that we can drop the connectedness assumption on $G$.

• For proper actions (Theorem \ref{reduced_forms_for_proper_actions}) we rely heavily on the Sjamaar--Lerman--Bates theory of stratified reduction. In effect, we show that $\Phi\inv(0)/G$ carries a \emph{global}, diffeological reduced form which induces theirs on each stratum. The interesting question whether all `Sjamaar forms' \cite{Sjamaar:2005} similarly globalize was discussed in \cite[3.41]{Watts:2012}, on which our proof is based. Also of future interest is the question whether the strata themselves are diffeological `Klein strata' \cite{Gurer:2023}. Note that if $G$ is compact, all actions are proper, so reduced forms always exist, and \eqref{t} is unconditionally a \emph{parasymplectomorphism} (see \ref{two_settings}).

In a final \emph{applications} section, we show that $\IND HGY$ carries a reduced form also when $H$~is not closed (Corollary \ref{reduced_form_on_induced_spaces}), and we obtain a Kummer--Marsden--Satzer isomorphism, $(T^*G)/\!\!/H = T^*(G/H)$, whenever $H$ is dense in $G$ --- e.g. when $G$ is the 2-torus and $H$ an irrational winding (Theorem~\ref{diffeological_KMS}).

\section*{Notation, conventions, and background}

This paper adopts \cite{Ratiu:2022}'s concise notation for the translation of tangent and cotangent vectors to a Lie group: for fixed $g,q\in G$,
\begin{equation}
	\label{concise}
	\gathered T_qG             \\\vspace{-3pt}       v       \endgathered
	\gathered  \,\to\,        \\\vspace{-3pt}   \,\mapsto\,  \endgathered
	\gathered T_{gq}G         \\\vspace{-3pt}      gv,       \endgathered
	\quad\qquad\text{resp.}\qquad\quad
	\gathered T^*_qG         \\\vspace{-3pt}       p         \endgathered
	\gathered  \,\to\,        \\\vspace{-3pt}  \,\mapsto\,   \endgathered
	\gathered T^*_{gq}G     \\\vspace{-3pt}      gp          \endgathered
\end{equation}
will denote the derivative of $q\mapsto gq$, respectively its contragredient, so that $\<gp,v\>=\<p,g\inv v\>$. Likewise we define $vg$ and $pg$ with 
$\<pg,v\>=\<p,vg\inv\>$. We resolve any operator precedence ambiguity using a lower dot, as in e.g.~\eqref{right_invariant_form}.

We take a \emph{Hamiltonian $G$-space} to mean the triple $(X, \omega, \Phi)$ of a manifold $X$ on which $G$ acts, a $G$-invariant symplectic form $\omega$ on $X$, and an $\text{Ad}^*$-equivariant moment map $\Phi:X\to\LG^*$. Our convention is $\i_{Z_X}\omega = - d\<\Phi(\cdot),Z\>$ with $Z_X(x)=\left.\frac d{dt}\e{tZ}(x)\right|_{t=0}$, and we recall the cardinal consequences \cite{Marsden:1974, Smale:1970}
\begin{samepage}
	\begin{equation}
	   \label{cardinal}
	   \text{(a) }\Ker(D\Phi(x)) = \LG(x)^\omega
	   \qquad\quad
	   \text{(b) }\Im (D\Phi(x)) = \ann(\LG_x).
	\end{equation}
	The first is the orthogonal relative to $\omega$ of the tangent space $\LG(x)$ 
	to the orbit $G(x)$; the second is the annihilator in $\LG^*$ of the stabilizer 
	Lie subalgebra $\LG_x\subset\LG$.
\end{samepage}

Regarding diffeology, we shall rely throughout on \cite{Souriau:1985a,Iglesias-Zemmour:2013} which the reader is assumed to have at hand; thus we assume familiarity with such notions as \emph{plot}, \emph{smooth map}, \emph{induction}, \emph{subduction}, \emph{subset diffeology}, \emph{quotient diffeology}. In addition we will make repeated use of the following unpublished observation from \cite[1.2.18]{Iglesias:1985} and \cite[2.10]{Watts:2012}:

\begin{prop}
	\label{restricted_subduction}
	Let $s:X\to Y$ be a subduction between diffeological spaces. Endow some $B\subset Y$ and its preimage $A=s\inv(B)$ with their subset diffeologies. Then $s_{|A}:A\to B$ is also a subduction.
\end{prop}

\begin{proof}
	For $A$ and $B$ to have subset diffeologies means that $A\overset{i}{\hookrightarrow}X$ and $B\overset{j}{\hookrightarrow}Y$ are inductions, hence in particular smooth. Therefore $s\circ i$ $(=j\circ s_{|A})$ is smooth,
	\begin{equation}
		\label{subquotient}
		\begin{tikzcd}[row sep=large,column sep=large,every label/.append style={font=\small}]
			&
			& A
			  \ar[r,hook,"i"]
			  \ar[d,swap,"s_{|A}"]
			& X
			  \ar[d,"s"]
			\\
			  u\in V
			  \ar[r,dashed,hook]
			  \ar[urr,dashed,"Q"]
			& U
			  \ar[r,"P"]
			& B
			  \ar[r,hook,"j"]
			& Y\rlap{,}
		\end{tikzcd}
	\end{equation}
	whence $s_{|A}$ is smooth by the universal property of inductions \cite[1.34]{Iglesias-Zemmour:2013}. Now to say that $s_{|A}$ is a \emph{subduction} means: given a plot $P:U\to B$ and $u\in U$, there are an open $V\ni u$ and a plot (``local lift'') $Q:V\to A$ such that \eqref{subquotient} commutes \cite[\nolinebreak 1.48]{Iglesias-Zemmour:2013}. To apply this, note that by the same token, $s$ being a subduction and $j\circ P$ a plot of $Y$ give us $V$ and a plot $R:V\to X$ such that $j\circ P_{|V} = s\circ R$. Moreover $j\circ P$ taking values in $B$ implies that $R$ takes values in $A$; so we have $R=i\circ Q$ for a map $Q$ as indicated, which is smooth by the same universal property. Now $j\circ P_{|V} = j\circ s_{|A}\circ Q$, whence (as desired) $P_{|V}=s_{|A}\circ Q$.
\end{proof}

\part{Frobenius reciprocity}
\section{Reduced Hamiltonian $G$-spaces and 2-forms}
Let $G$ be a Lie group and $(X,\omega,\Phi)$ a Hamiltonian $G$-space. While the resulting \emph{reduced space}
\begin{equation}
	\label{reduced_space}
	X/\!\!/G := \Phi\inv(0)/G
\end{equation}
need not be a manifold, we may (and will) endow it with the \emph{subquotient diffeology}: first put on $\Phi\inv(0)$ the subset diffeology, and then on \eqref{reduced_space} the quotient diffeology \cite[1.33,\,1.50]{Iglesias-Zemmour:2013}. (By \eqref{restricted_subduction} it is equivalent to first put on $X/G$ the quotient diffeology, and then on \eqref{reduced_space} the subset diffeology of~that.) So we have a ready notion of differential form on $X/\!\!/G$ \cite[6.28]{Iglesias-Zemmour:2013}, and the question naturally arises whether (or when) it carries a reduced 2-form:

\begin{samepage}
	\begin{defi}
		\label{reduced_2-form}
		We say that $X/\!\!/G$ \emph{carries a reduced $2$-form}, if there is on it a (diffeological) 2-form $\omega_{X/\!\!/G}$ such that $j^*\omega=\pi^*\omega_{X/\!\!/G}$, where $j$ and $\pi$ are the natural maps in
	\begin{equation}
		\label{reduction_diagram}
	    \begin{tikzcd}
	       \Phi\inv(0) \rar[hook]{j}\dar[swap]{\pi} & X \\
	       X/\!\!/G\rlap{.}
	    \end{tikzcd}
	\end{equation}
	We note that if $\omega_{X/\!\!/G}$ exists, then it is unique (and closed) by \cite[6.39]{Iglesias-Zemmour:2013}.
	\end{defi}
\end{samepage}

\begin{rems}
	\label{two_settings}
	We regard the existence of $\omega_{X/\!\!/G}$ as an important open problem in diffeology, so far only solved, to our knowledge, in two familiar settings. First, if the $G$-action on $\Phi\inv(0)$ is \emph{free and proper}, then \eqref{reduction_diagram} consists of manifolds, and $\omega_{X/\!\!/G}$ exists as an ordinary (hence also diffeological:~\cite[2.14]{Karshon:2016}) 2-form. This is the Marsden--Weinstein setting \cite{Marsden:1974}. Secondly, if that action\- is \emph{still proper} but only \emph{locally free} (i.e., all infinitesimal stabilizers $\LG_x$ are zero), then $X/\!\!/G$ is an (effective) orbifold and carries an `orbifold 2-form': \cite[p.\,337]{Cushman:1997}; note that $\Phi\inv(0)$ is still a manifold, as (\ref{cardinal}b) still ensures that $0$ is a regular value of $\Phi$. Now as proved in \cite[App.\,A]{Karshon:2016}, when orbifolds are regarded as diffeological spaces, `orbifold forms' define diffeological forms and conversely. So $\omega_{X/\!\!/G}$ exists also in this orbifold setting.

Note that whereas the Marsden--Weinstein 2-form is \emph{symplectic}, there is no consensus beyond their setting on when to call a closed 2-form (or diffeological space) symplectic: see the discussion in \cite[p.\,1310]{Iglesias-Zemmour:2016}, which we shall follow in calling $X/\!\!/G$ merely \emph{parasymplectic} when $\omega_{X/\!\!/G}$ exists.
\end{rems}

In this paper we are mainly concerned with the following two reductions:

\begin{s-intertwiner}
	\label{s-intertwiner}
	This is
	\begin{equation}
	   \label{symplectic_hom}
		   \begin{aligned}
			   \Hom_G(X_1,X_2)&:=(X_1^-\times X_2^{\vphantom-})/\!\!/G\\
			   &\phantom{:}=\Phi\inv(0)/G,
		   \end{aligned}
	   	\end{equation}
	where $(X_i,\omega_i,\Phi_i)$ are Hamiltonian $G$-spaces and $X_1^-$ is short for $(X_1,-\omega_1,-\Phi_1)$; so the product here has $G$-action $g(x_1,x_2)=(g(x_1),g(x_2))$, 2-form $\omega_2-\omega_1$ and moment map $\Phi(x_1,x_2)=\Phi_2(x_2)-\Phi_1(x_1)$. Note that \eqref{symplectic_hom} boils down to $X_2/\!\!/G$ when $(X_1,\omega_1,\Phi_1)=(\{0\},0,0)$, so asking when it carries a reduced 2\nobreakdash-form includes the original problem about \eqref{reduced_space}. (More generally, \cite{Guillemin:1982} took for $X_1$ a coadjoint orbit $G(\mu)$ and observed that \eqref{symplectic_hom} then boils down to the space $\Phi_2\inv(\mu)/G_\mu$ of \cite{Marsden:1974}: this is their famous ``shifting trick''.)
\end{s-intertwiner}

\begin{s-induced}
	\label{s-induced}
	This is
	\begin{equation}
		\begin{aligned}
		   \label{induced_manifold}
		   \IND HGY&:=(T^*G\times Y)/\!\!/H\\
		   &\phantom{:}= \psi\inv(0)/H,
		\end{aligned}
	\end{equation}
	where $H\subset G$ is a closed subgroup, $(Y,\omega_Y,\Psi)$ a Hamiltonian $H$-space, and $L:=T^*G\times Y$ is the Hamiltonian $G\times H$-space with action $(g,h)(p,y)=(gph\inv, h(y))$, 2-form $\omega_L=d\varpi_{T^*G} + \omega_Y$ ($\varpi_{T^*G}$ is the canonical 1-form), and resulting moment map $\phi\times\psi:L\to\LG^*\times\LH^*$,
	\begin{equation}
	   \label{phi_and_psi}
	   \left\{
	   \begin{array}{lll}
	      \rlap{$\phi$}\phantom{\psi}(p,y) & \!\!= & \!\!pq\inv\\[.5ex]
	      \psi(p,y) & \!\!= & \!\!\Psi(y)- q\inv p\subh
	   \end{array}
	   \right.
	   \rlap{\qquad$(p\in T^*_qG)$}
	\end{equation}
	(notation \ref{concise}); recall that $\varpi_{T^*G}(\delta p) = \<p,\delta q\>$ where $\delta q\in T_qG$ is the image of $\delta p \in T_p(T^*G)$ by the derivative of the projection $\pi:T^*G \to G$. As~detailed in \cite[1.2]{Ratiu:2022} this fits the Marsden--Weinstein setting \eqref{two_settings}, so \eqref{induced_manifold} is a manifold carrying a reduced 2-form $\omega_{L/\!\!/H}$. Moreover it has a residual $G$-action and moment map $\Phi_{L/\!\!/H}$ (deduced from $\phi$), so that altogether $(\IND HGY,\omega_{L/\!\!/H},\Phi_{L/\!\!/H})$ is a new Hamiltonian $G$-space; one calls it \emph{induced by $Y$\textup, from $H$}.
	
	As observed in \cite[§2.3]{Guillemin:2005a} and \cite[6.10]{Sjamaar:2005}, important examples of induced $G$-spaces are the \emph{normal forms} of Marle--Guillemin--Sternberg, to which results like \eqref{symplectic_frobenius} will apply.
\end{s-induced}

\section{Symplectic Frobenius reciprocity}

We continue with $H$ a closed subgroup of $G$, and write $\Res^G_H$ for the restriction functor from Hamiltonian $G$- to Hamiltonian $H$-spaces, i.e., the $G$-action is restricted to $H$ and the moment map is composed with the projection $\LG^*\to\LH^*$. The following proves the conjecture of \cite[3.5]{Ratiu:2022}.

\begin{theo}
   \label{symplectic_frobenius}
   If $X$ is a Hamiltonian $G$-space and $Y$ a Hamiltonian $H$-space\textup, then we have an isomorphism of reduced spaces
   \begin{equation}\
	  \label{symplectic_frobenius_equality}
      \Hom_G(X,\IND HGY)=\Hom_H(\Res^G_HX,Y),
   \end{equation}
   i.e.~there is a \textup(diffeological\textup) diffeomorphism $t$ from left to right. Moreover\textup, if one side carries a reduced $2$-form\textup, then so does the other\textup, and $t$ relates the $2$-forms.
\end{theo}

\begin{proof}
We follow \cite[3.4]{Ratiu:2022} which obtained $t$ as a mere bijection. Write $\Phi$, $\Psi$ for the moment maps of $X$ and $Y$, and unravel the combination of \eqref{symplectic_hom} and \eqref{induced_manifold}: we get that the sides of \eqref{symplectic_frobenius_equality} are respectively $(M/\!\!/H)/\!\!/G$ and $N/\!\!/H$, where
\begin{equation}
	\label{M_and_N}
	M=X^-\times T^*G\times Y,
	\qquad\text{resp.}\qquad
	N=X^-\times Y
\end{equation}
are the Hamiltonian $G\times H$-space with action $(g,h)(x,p,y)=(g(x),gph\inv,h(y))$, 2-form $\omega_M=\omega_Y + d\varpi_{T^*G} - \omega_X$ and moment map $\phi_M\times\psi_M: M\to\LG^*\times\LH^*$, resp.~$H$\nobreakdash-space with diagonal $H$-action, 2-form $\omega_N=\omega_Y-\omega_X$ and moment map $\psi_N:N\to\LH^*$, where
   \begin{equation}
	   \label{moments_M_and_N}
	      \left\{
	      \begin{array}{rll}
	         \rlap{$\phi_M$}\phantom{\psi_M}(x,p,y) & \!\!= 
	         & \!\! pq\inv-\Phi(x)\\[.5ex]
	         \psi_M(x,p,y) & \!\!= & \!\!\Psi(y)- q\inv p\subh\\[.5ex]
	   	     \psi_N(x,y) & \!\!= & \!\!\Psi(y)-\Phi(x)\subh
	      \end{array}
	      \right.
	      \rlap{\qquad$(p\in T^*_{q}G)$.}\\
   \end{equation}
Define a submersion $r:M\to N$ by $r(x,p,y)=(q\inv(x),y)$ where $p\in T^*_{q}G$, and consider the following commutative diagram, where we have written $j_1,j_2,j_3$ and $\pi_1,\pi_2,\pi_3$ for the inclusions and projections involved in constructing reduced spaces \eqref{reduction_diagram}, $\Phi_{M/\!\!/H}$ for the moment map of the residual $G$-action on $M/\!\!/H$, and $j,\pi$ for the obvious inclusion and restriction (of $\pi_1$):
\begin{equation}
	\label{symplectic_diagram}
    \begin{tikzcd}[sep=small, arrows={line width=1.2*rule_thickness}]
       M
       \ar[rrr,"r"]
       &&&
       N
       \\
       &&
       \psi_M\inv(0)
       \ar[ull,hook',swap,"j_1" inner sep=0.2ex]
       \ar[dd,"\pi_1"]
       \\
       &&&
       (\phi_M\times\psi_M)\inv(0)
       \ar[ul,hook',swap,"j" inner sep=0.25ex]
       \ar[rrrrrrr,"s",dashed]
       \ar[dd,"\pi"]
       &&&&&&&
       \psi_N\inv(0)
       \ar[uulllllll,hook',swap,"j_3" inner sep=0.25ex]
       \ar[dddddd,"\pi_3"]
       \\
       &&
       M/\!\!/H
       \\
       &&&
       \Phi_{M/\!\!/H}\inv(0)
       \ar[ul,hook',swap,"j_2" inner sep=0.25ex]
       \ar[dddd,"\pi_2"]
       \\
       \\
       \\
       \\
       &&&
       (M/\!\!/H)/\!\!/G
       \ar[rrrrrrr,"t",dashed]
       &&&&&&&
       N/\!\!/H.
    \end{tikzcd}
\end{equation}
Also define a right inverse immersion $r':N\to M$ by $r'(x,y)=(x,\Phi(x),y)$, where we identify $\LG^*$ with the cotangent space of $G$ at the identity. As shown in \cite[\nolinebreak 3.8]{Ratiu:2022}, $r$ and $r'$ send $(\phi_M\times\psi_M)\inv(0)$ to $\psi_N\inv(0)$ and conversely; so they induce a map $s$ as indicated in \eqref{symplectic_diagram}, and a right inverse $s'$ of it. Next \cite[\nolinebreak 3.9]{Ratiu:2022} shows that $s$ sends $G\times H$-orbits to $H$-orbits; conversely $s'$ sends $H$-orbits to orbits of the diagonal $\mathrm{diag}(H)\subset G\times H$. It follows that $s$ and $s'$ descend to a bijection $t$, as indicated in \eqref{symplectic_diagram}, and its inverse $t\inv$.

To see that $t$ is smooth, note first that $r$ is of course diffeologically smooth \cite[4.3C, end of proof]{Iglesias-Zemmour:2013}. Also, by construction in \eqref{reduced_space} of the subset diffeologies on the zero levels of $\psi_M$, $\Phi_{M/\!\!/H}$, $\psi_N$ and resulting quotient diffeologies on $M/\!\!/H$, $(M/\!\!/H)/\!\!/G$ and $N/\!\!/H$, the injections $j_1,j_2,j_3$ are \emph{inductions}, hence smooth, and the projections $\pi_1,\pi_2,\pi_3$ are \emph{subductions}, hence also smooth \cite[1.29, 1.36, 1.46, 1.50]{Iglesias-Zemmour:2013}. Now give $(\phi_M\times\psi_M)\inv(0)$ its subset diffeology inside $\psi_M\inv(0)$: then $j$ is another induction, and $\pi$ is a subduction by \eqref{restricted_subduction}. In particular $r\circ j_1\circ j$ is smooth \cite[1.15]{Iglesias-Zemmour:2013}. As this is just $j_3\circ s$, it follows by the universal property of inductions \cite[1.34]{Iglesias-Zemmour:2013} that $s$ is smooth. Then $\pi_3\circ s$ is smooth, and as this is just $t\circ\pi_2\circ\pi$, it follows by the universal property of subductions \cite[1.51]{Iglesias-Zemmour:2013} that $t$ is smooth.

To see that $t\inv$ is smooth, note likewise that $r'$ is diffeologically smooth. Hence so is $r'\circ j_3$. As this is just $j_1\circ j\circ s'$, the universal property of inductions gives smoothness of $s'$. Then $\pi_2\circ\pi\circ s'$ is smooth, and as this is just $t\inv\circ\pi_3$, the universal property of subductions gives smoothness of $t\inv$. 

So $t$ is a diffeomorphism, as claimed. Next, assume that both sides of \eqref{symplectic_frobenius_equality} carry reduced 2-forms, $\omega_{(M/\!\!/H)/\!\!/G}$ and $\omega_{N/\!\!/H}$. We must prove
\begin{equation}
	\label{desired_pull_back}
	\omega_{(M/\!\!/H)/\!\!/G}=t^*\omega_{N/\!\!/H}.
\end{equation}
To this end, we claim it is enough to show, not quite $\omega_M=r^*\omega_N$ but
\begin{equation}
	\label{key_pull_back}
	j^*j_1^*\omega_M=j^*j_1^*r^*\omega_N.
\end{equation}
Indeed, by commutativity of \eqref{symplectic_diagram} and definition \eqref{reduced_2-form} of a reduced 2-form (of which $M/\!\!/H$ always carries one, since it is a Marsden--Weinstein reduction \eqref{s-induced}), we get that \eqref{key_pull_back} has left-hand side $j^*\pi_1^*\omega_{M/\!\!/H}=\pi^*j_2^*\omega_{M/\!\!/H}=\pi^*\pi_2^*\omega_{(M/\!\!/H)/\!\!/G}$ and right-hand side $s^*j_3^*\omega_N=s^*\pi_3^*\omega_{N/\!\!/H}=\pi^*\pi_2^*t^*\omega_{N/\!\!/H}$. Since pull-back $\pi^*\pi_2^*$ by the subduction $\pi_2\circ\pi$ is injective \cite[6.39]{Iglesias-Zemmour:2013}, we conclude that \eqref{key_pull_back} indeed implies \eqref{desired_pull_back}.

Now \eqref{key_pull_back} is a genuinely diffeological relation to prove, because (unlike $r,j_1$), $j$ need not be a smooth map between \emph{manifolds}. By definition \cite[6.28, 6.33]{Iglesias-Zemmour:2013}, \eqref{key_pull_back} means that its two sides coincide \emph{after pull-back by any plot}, $P$, of the (subset) diffeology of $(\phi_M\times\psi_M)\inv(0)$, i.e., $(j_1\circ j\circ P)^*\omega_M=(j_1\circ j\circ P)^*r^*\omega_N$. But then $F:=j_1\circ j\circ P$ is an ordinary smooth map $U\to M$, $U$ open in some~$\RR^n$, taking values in $(\phi_M\times\psi_M)\inv(0)$ \cite[1.33]{Iglesias-Zemmour:2013}. So \eqref{key_pull_back} will be proved if we show
\begin{equation}
	\label{auxiliary_pull_back}
	F^*\omega_M=F^*r^*\omega_N
\end{equation}
\emph{for every ordinary smooth map\textup, $F:U\to M$\textup, taking values in $(\phi_M\times\psi_M)\inv(0)$}. To this end, fix such a map $F$, write $F(u)$ as $m=(x,p,y)$, and regard also the base point $q$ of $p\in T^*_qG$ and the subsequent variables 
\begin{equation}
	\mu=pq\inv,\qquad
	\bar\mu=\Phi(x),\qquad
	\bar x=q\inv(x),\qquad
	n = r(m) = (\bar x,y)
\end{equation}
as smooth functions of $m$ and therefore $u$. Then derivatives of these functions map each vector $\delta u\in T_uU$ to vectors $\delta m$, $\delta x$, $\delta p$, $\delta y$, $\delta q$, $\delta\mu$, $\delta\bar\mu$, $\delta\bar x$, $\delta n$. Writing also $Z:=\delta q.q\inv$ ($=\Theta(\delta q)$ where $\Theta$ is a Maurer--Cartan 1-form), we~have $\delta q=Zq$ and thus
\begin{align}
	\label{action_map_derivative}
	\delta\bar x
	&=\frac{\partial[q\inv(x)]}{\partial q}(\delta q)
	+\frac{\partial[q\inv(x)]}{\partial x}(\delta x)\notag\\
	&=\Bigl.\frac d{dt}(\e{tZ}q)\inv(x)\Bigr|_{t=0} + q\inv{}_*(\delta x)\\
	&=\Bigl.\frac d{dt}\,q\inv(\e{-tZ}(x))\Bigr|_{t=0} + q\inv{}_*(\delta x)\notag\\
	&=q\inv{}_*(\delta x-Z_X(x)).\notag
\end{align}
With this in hand, we compute
\begin{align}
	\label{pull_back_computation}
	(F^*r^*\omega_N)(\delta u,\delta'u)
	&=\omega_N(\delta n,\delta'n)\notag\\
	&=\omega_Y(\delta y,\delta'y) - \omega_X(\delta\bar x,\delta'\bar x)\notag\\
	&=\omega_Y(\delta y,\delta'y)
	- \omega_X(\delta x-Z_X(x),\delta'x-Z'_X(x))\notag\\
	&=\omega_Y(\delta y,\delta'y)
	+ \omega_X(Z_X(x),\delta'x)
	- \omega_X(Z'_X(x),\delta x)\notag\\
	&\hspace{2.15cm}
	- \omega_X(Z_X(x),Z'_X(x))
	- \omega_X(\delta x,\delta'x)\notag\\
	&=\omega_Y(\delta y,\delta'y)
	+ \<\delta\bar\mu,Z'\>
	- \<\delta'\bar\mu,Z\>
	+\<\bar\mu,[Z,Z']\>\\
	&\hspace{2.15cm}
	- \omega_X(\delta x,\delta'x)\notag\\
	&=\omega_Y(\delta y,\delta'y)
	+ \<\delta\mu,Z'\>
	- \<\delta'\mu,Z\>
	+\<\mu,[Z,Z']\>\notag\\
	&\hspace{2.15cm}
	- \omega_X(\delta x,\delta'x)\notag\\
	&=\omega_Y(\delta y,\delta'y)
	+ d\varpi_{T^*G}(\delta p,\delta'p)
	- \omega_X(\delta x,\delta'x)\notag\\
	&=\omega_M(\delta m,\delta'm)\notag\\
	&=(F^*\omega_M)(\delta u,\delta'u).\notag
\end{align}
Here the third equality is by \eqref{action_map_derivative} and $G$-invariance of $\omega_X$; the fifth is because $\Phi$ is an (equivariant) moment map and by \cite[11.17$\sharp$]{Souriau:1970}; the sixth is because $F(U)\subset\phi_M\inv(0)$; and the seventh is the formula for $d\varpi_{T^*G}$ found in \cite[4.4.1]{Abraham:1978}, also obtainable by writing $\varpi_{T^*G}(\delta p)=\<pq\inv,\delta q.q\inv\>=\<\mu,\Theta(\delta q)\>$ and using the Maurer--Cartan formula $d\Theta(\delta q,\delta'q)=\delta'q\delta[q\inv] -\delta q\delta'[q\inv]=[\Theta(\delta q),\Theta(\delta'q)]$ \cite[\nolinebreak III.3.14]{Bourbaki:1972}. Thus we have proved \eqref{auxiliary_pull_back} and thereby also \eqref{key_pull_back} and \eqref{desired_pull_back}.

Finally, assume merely that \emph{one} reduced 2-form exists, $\omega_{(M/\!\!/H)/\!\!/G}$ or $\omega_{N/\!\!/H}$. Then we can \emph{define} the other by \eqref{desired_pull_back}, and a straightforward diagram chase in \eqref{symplectic_diagram} (using again \eqref{key_pull_back}, whose proof involved neither reduced 2-form) shows that we have $\pi_3^*\omega_{N/\!\!/H} = j_3^*\omega_N$ (resp.~$\pi_2^*\omega_{(M/\!\!/H)/\!\!/G} = j_2^*\omega_{M/\!\!/H}$), as desired. 
\end{proof}

\begin{rema}
	An initially simpler proof results if we replace \eqref{desired_pull_back} by $\omega_{N/\!\!/H}=(t\inv)^*\omega_{(M/\!\!/H)/\!\!/G}$ and (\ref{key_pull_back}--\ref{auxiliary_pull_back}) by $\omega_N=r'^*\omega_M$. However, the full force of \eqref{key_pull_back} is still required in the last 4 lines (after ``Finally'').
\end{rema}

\section{Reduced prequantum $G$-spaces and 1-forms}\label{prequantum}

\setlength\epigraphwidth{.56\textwidth}
\renewcommand{\epigraphsize}{\footnotesize}
\setlength\epigraphrule{0pt}
\epigraph{One might now say ``symplectic geometry is all geometry,'' but I prefer to formulate it in a more geometrical form: \emph{contact geometry is all geometry}.}{---V.~I.~Arnol'd \cite{Arnold:2000}}

\noindent
Besides Hamiltonian $G$-spaces, \cite{Ratiu:2022} makes the case that one should also induce and reduce \emph{prequantum $G$-spaces}, whose theory \cite[§18]{Souriau:1970} is recalled in \cite[\nolinebreak §5]{Ratiu:2022}. Briefly, such a space is the pair $(\tilde X,\varpi)$ of a manifold $\tilde X$ on which $G$ acts, and a $G$-invariant contact $1$-form $\varpi$ whose Reeb vector field \cite[p.\,17]{Reeb:1952a} generates a free action of the circle group $\TT=\mathrm{U}_1$. Then the \emph{prequantum moment map} $\Phi:\tilde X\to\LG^*$,
\begin{equation}
   \label{prequantum_moment}
	\<\Phi(\tilde x),Z\>=\varpi(Z_{\tilde X}(\tilde x))
	\rlap{\qquad\qquad$(Z\in \LG)$,}
\end{equation}
makes $(\tilde X,d\varpi,\Phi)$ a presymplectic Hamiltonian $G$-manifold in the sense of \cite{Lin:2019}, with null leaves the $\TT$-orbits. The $G$-action commutes with $\TT$, so it descends to the leaf space $X:=\tilde X/\TT$; likewise $d\varpi$ and $\Phi$ descend to a symplectic form $\underline\omega$ and moment map $\underline\Phi$ which make $(X,\underline\omega,\underline\Phi)$ a Hamiltonian $G$-space; one says that $\tilde X$ \emph{prequantizes} $X$.

Now as in \eqref{reduced_space} we can consider the (now \emph{prequantum}) \emph{reduced space}
\begin{equation}
	\label{prequantum_reduced_space}
	\tilde X/\!\!/G:=\Phi\inv(0)/G,
\end{equation}
endow it with the subquotient diffeology, and make the
\begin{defi}
	\label{reduced_1-form}
	We say that $\tilde X/\!\!/G$ \emph{carries a reduced $1$-form}, if it admits a (necessarily unique, diffeological) 1-form $\varpired$ such that $j^*\varpi=\pi^*\varpired$,~where
	\begin{equation}
		\label{prequantum_reduction_diagram}
	    \begin{tikzcd}
	       \Phi\inv(0) \rar[hook]{j}\dar[swap]{\pi} & \tilde X \\
	       \tilde X/\!\!/G\rlap{.}
	    \end{tikzcd}
	\end{equation}
\end{defi}

 \begin{rema}
 	\label{two_prequantum_settings}
	As in \eqref{two_settings}, the classic setting where we have a reduced 1-form is when the $G$-action on $\Phi\inv(0)$ is \emph{free and proper}: then again \eqref{prequantum_reduction_diagram} consists of manifolds and $\varpired$ exists as an ordinary 1-form \cite[Thm~2]{Loose:2001}, \cite[5.5]{Ratiu:2022}. More generally, \cite[\nolinebreak 5.10]{Karshon:2016} implies that a diffeological $\varpired$ exists whenever the action of the identity component $G^\mathrm{o}$ on $\Phi\inv(0)$ is \emph{locally free and proper}. 	
 \end{rema}

We need the following prequantum versions of \eqref{s-intertwiner} and \eqref{s-induced}:

\begin{p-intertwiner}
	\label{p-intertwiner}
	This is
	\begin{equation}
	   \begin{aligned}
		   \label{prequantum_hom}
		   \Hom_G(\tilde X_1, \tilde X_2)&:=
		   (\tilde X_1^-\boxtimes\tilde X_2^{\vphantom-})/\!\!/G\\
		   &\phantom{:}=\Phi\inv(0)/G,
	   \end{aligned}
	\end{equation}
	where $(\tilde X_i,\varpi_i)$ are prequantum $G$-spaces and $\tilde X_1\boxtimes\tilde X_2$ denotes their \emph{prequantum product} \cite[18.52]{Souriau:1970}, i.e.~the quotient of $\tilde X_1\times\tilde X_2$ by the action of the antidiagonal subgroup $\Delta=\{(z\inv,z):z\in\TT\}$ of $\TT^2$, with 1-form and $G$-action obtained by passage to the quotient from $\varpi_1+\varpi_2$ and the diagonal $G$-action. Moreover $\tilde X_1^-$ is short for $(\tilde X_1,-\varpi_1)$, so \eqref{prequantum_hom} uses instead the 1-form $\varpi_2-\varpi_1$, the $\Delta$-action $(z(\tilde x_1),z(\tilde x_2))$, and by \eqref{prequantum_moment}, the moment map $\Phi(\Delta(\tilde x_1,\tilde x_2))=\Phi_2(\tilde x_2)-\Phi_1(\tilde x_1)$.
\end{p-intertwiner}

\begin{p-induced}
	\label{p-induced}
	This is
	\begin{equation}
	   \label{prequantum_induction}
	   \begin{aligned}
		   \IND HG{\tilde Y}&:=(T^*G\times\tilde Y)/\!\!/H\\
		   &\phantom{:}= \psi\inv(0)/H,
	   \end{aligned}
	\end{equation}
	where $H\subset G$ is a closed subgroup, $(\tilde Y,\varpi_{\tilde Y})$ is a prequantum $H$-space, and $\smash{\tilde L}=T^*G\times\smash{\tilde Y}$ is the prequantum $G\times H$-space with action $(g,h)(p,\tilde y)=(gph\inv, h(\tilde y))$ and $1$-form $\varpi_{T^*G}+\smash{\varpi_{\tilde Y}}$; the $\TT$-action is $z(p,\tilde y)=(p,z(\tilde y))$, and the moment map \eqref{prequantum_moment}, $\phi\times\psi:\smash{\tilde L}\to\LG^*\times\LH^*$, is given by formulas \eqref{phi_and_psi} with a tilde on every $y$. As with \eqref{s-induced}, one verifies that \eqref{prequantum_induction} is a manifold with a reduced 1-form $\varpi_{\smash{\tilde L}/\!\!/H}$, and naturally a prequantum $G$-space prequantizing \eqref{induced_manifold}.
	
	Also as with \eqref{s-induced}, one can see that the local models of Lerman and Willett \cite[4.1]{Lerman:2001} are induced $G$-spaces.
\end{p-induced}

\section{Prequantum Frobenius reciprocity}\label{prequantum_Frobenius}

With $H\subset G$ still a closed subgroup, write $\Res^G_H$ for the restriction functor from prequantum $G$- to prequantum $H$-spaces, i.e., the $G$-action gets restricted to $H$.

\begin{theo}
   \label{prequantum_frobenius}
   If $\tilde X$ is a prequantum $G$-space and $\tilde Y$ a prequantum $H$-space\textup, then we have an isomorphism of reduced spaces
   \begin{equation}\
	  \label{prequantum_frobenius_equality}
      \Hom_G(\tilde X,\IND HG{\tilde Y})=\Hom_H(\Res^G_H\tilde X,\tilde Y),
   \end{equation}
   i.e.~there is a \textup(diffeological\textup) diffeomorphism $t$ from left to right. Moreover\textup, if one side carries a reduced $1$-form\textup, then so does the other\textup, and $t$ relates the $1$-forms.
\end{theo}

\begin{proof}
	The proof closely parallels that of \eqref{symplectic_frobenius}, and we shall mainly emphasize what's new. First we see again that the sides of \eqref{prequantum_frobenius_equality} are $(M/\!\!/H)/\!\!/G$ and $N/\!\!/H$, with \eqref{M_and_N} now replaced by the \emph{prequantum} $G\times H$-space and $H$-space
	\begin{equation}
		\label{prequantum_M_and_N}
		M=\tilde X^-\boxtimes (T^*G\times\tilde Y),
		\qquad\text{resp.}\qquad
		N=\tilde X^-\boxtimes\tilde Y.
	\end{equation}
	In more detail: Define $\check M$ and $\check N$ by these two formulas with $\times$ in place of $\boxtimes$, and endow them with the 1-forms $\varpi_{\check M}=\varpi_{\tilde Y}+\varpi_{T^*G}-\varpi_{\tilde X}$ and $\varpi_{\check N}=\varpi_{\tilde Y}-\varpi_{\tilde X}$. Then $M$ and $N$ are the orbit spaces of $\Delta$-actions \eqref{p-intertwiner}, with 1-forms $\varpi_M$, $\varpi_N$ and group actions obtained by passage to the quotient from $\varpi_{\check M}$, $\varpi_{\check N}$ and
   \begin{equation}
	   (g,h)(\tilde x,p,\tilde y)=(g(\tilde x),gph\inv,h(\tilde y)),
	   \ \quad\text{resp.}\ \quad
	   h(\tilde x,\tilde y)=(h(\tilde x),h(\tilde y)).
   \end{equation}
   The resulting moment maps \eqref{prequantum_moment}, $\phi_M\times\psi_M: M\to\LG^*\times\LH^*$ and $\psi_N:N\to\LH^*$, or rather their pull-backs $\phi_{\check M}\times\psi_{\check M}$ and $\psi_{\check N}$ to $\check M$ and $\check N$, are given by the same formulas \eqref{moments_M_and_N} as before, except with a tilde over every $x$ and $y$.
   
   Now define a submersion $\check r:\check M\to\check N$ and right inverse immersion $\check r':\check N\to\check M$ by $\check r(\tilde x,p,\tilde y)=(q\inv(\tilde x),\tilde y)$ where $p\in T^*_qG$, respectively $\check r'(\tilde x, \tilde y) = (\tilde x, \Phi(\tilde x),\tilde y)$, where we identify $\LG^*$ with the cotangent space of $G$ at the identity. Clearly these send $\Delta$-orbits to $\Delta$-orbits, so they induce maps $r:M\to N$ and $r':N\to M$ which are smooth by \cite[5.9.6]{Bourbaki:1967}. At this point, \emph{verbatim} the same diagram \eqref{symplectic_diagram} and arguments about it establish as before that $r$ descends to a bijection $t:(M/\!\!/H)/\!\!/G\to N/\!\!/H$, and that $t$ is a diffeomorphism.

   Next, assume that both sides carry reduced 1-forms, $\varpi_{(M/\!\!/H)/\!\!/G}$ and $\varpi_{N/\!\!/H}$. We must show that $\varpi_{(M/\!\!/H)/\!\!/G}=t^*\varpi_{N/\!\!/H}$. Just as in \eqref{key_pull_back}, we see that this~boils down to proving
   \begin{equation}
	   \label{key_pull_back_prequantum}
	   j^*j_1^*\varpi_M=j^*j_1^*r^*\varpi_N.
   \end{equation}
   Now we know from \eqref{restricted_subduction} that the projection $\check M\to M$ induces a \emph{subduction} $(\phi_{\check M}\times\psi_{\check M})\inv(0)\to(\phi_{M}\times\psi_{M})\inv(0)$. Therefore \eqref{key_pull_back_prequantum} will follow if we show that its two sides coincide after pull-back by that subduction, i.e., that
   \begin{equation}
	   \label{upstairs_pull_back_prequantum}
	   \check\jmath^*\check\jmath_1^*\varpi_{\check M}=
	   \check\jmath^*\check\jmath_1^*\check r^*\varpi_{\check N}
   \end{equation}
   where $\check\jmath$ and $\check\jmath_1$ denote the two inclusions $(\phi_{\check M}\times\psi_{\check M})\inv(0)\longhookrightarrow\smash{\psi_{\check M}\inv(0)}\longhookrightarrow\check M$. And exactly as in \eqref{auxiliary_pull_back}, we see that \eqref{upstairs_pull_back_prequantum} will follow if we show
\begin{equation}
	\label{auxiliary_pull_back_prequantum}
	F^*\varpi_{\check M}=F^*\check r^*\varpi_{\check N}
\end{equation}
\emph{for every ordinary smooth map\textup, $F:U\to\check M$\textup, taking values in $(\phi_{\check M}\times\psi_{\check M})\inv(0)$}. So fix such a map $F$, write $F(u)$ as $\check m=(\tilde x,p,\tilde y)$, and regard also the base point $q$ of $p\in T^*_qG$ and $\check n = \check r(\check m) = (q\inv(\tilde x),\tilde y)$ as smooth functions of $\check m$ and $u$. Then again derivatives of these functions (and right translation by $q\inv$) map each $\delta u\in T_uU$ to vectors $\delta\check m$, $\delta\tilde x$, $\delta p$, $\delta\tilde y$, $\delta q$, $\delta\check n$ and an element $Z:=\delta q.q\inv\in\LG$; and computing as in \eqref{pull_back_computation}, we get
\begin{align}
	(F^*\check r^*\varpi_{\check N})(\delta u)
	&=\varpi_{\check N}(\delta\check n)\notag\\
	&=\varpi_{\tilde Y}(\delta\tilde y)
	-\varpi_{\tilde X}(\delta[q\inv(\tilde x)])\notag\\
	&=\varpi_{\tilde Y}(\delta\tilde y)
	-\varpi_{\tilde X}(\delta\tilde x - Z_{\tilde X}(\tilde x))\notag\\
	&=\varpi_{\tilde Y}(\delta\tilde y) + \<\Phi(\tilde x),Z\> 
	- \varpi_{\tilde X}(\delta\tilde x)\\
	&=\varpi_{\tilde Y}(\delta\tilde y) + \<pq\inv,Z\> 
	- \varpi_{\tilde X}(\delta\tilde x)\notag\\
	&=\varpi_{\tilde Y}(\delta\tilde y) + \varpi_{T^*G}(\delta p)
	- \varpi_{\tilde X}(\delta\tilde x)\notag\\
	&=\varpi_{\check M}(\delta\check m)\notag\\
	&=(F^*\varpi_{\check M})(\delta u)\notag
\end{align}
as desired. Here the third equality is as in \eqref{action_map_derivative}, the fourth is \eqref{prequantum_moment}, and the fifth is because $F(U)\subset\phi_{\check M}\inv(0)$.

Finally, assume merely that \emph{one} reduced 1-form exists, $\varpi_{(M/\!\!/H)/\!\!/G}$ or $\varpi_{N/\!\!/H}$. Then again we can define the other by $\varpi_{(M/\!\!/H)/\!\!/G}=t^*\varpi_{N/\!\!/H}$, and again chasing in \eqref{symplectic_diagram} and using \eqref{key_pull_back_prequantum} proves $\pi_3^*\varpi_{N/\!\!/H} = j_3^*\varpi_N$ (resp.~$\pi_2^*\varpi_{(M/\!\!/H)/\!\!/G} = j_2^*\varpi_{M/\!\!/H}$), as desired. 
\end{proof}

\section{Relation with unitary Frobenius reciprocity}

Our discussion so far has been purely geometrical, in the categories of Hamiltonian or prequantum $G$- and $H$-spaces; yet we motivated it by a ``corresponding'' story about unitary representations. In this section we recall the unitary story, then describe and exemplify how the correspondence is \emph{expected} to work. We do so under the assumptions, required for unitary Frobenius reciprocity and highest weight theory, that \emph{$G$ and $H$ are compact in} \ref{unitary_frobenius_section}--\ref{expectations}, and \emph{connected in} \ref{Borel-Weil}--\ref{noncompact}.

\subsection{Unitary Frobenius reciprocity}
\label{unitary_frobenius_section}
In the category \{unitary $G$-modules\}, objects are Hilbert spaces $(V,(\,\cdot\,|\,\cdot\,))$ with a continuous $G$-action by unitary operators; $\Hom_G(V,V')$ consists of all $G$-equivariant, bounded linear maps $V\to\nolinebreak V'$; restriction $\Res^G_H$ is self-explanatory; and \emph{induction} $\IND HG{}$ takes a unitary $H$-mod\-ule $W$ to the unitary $G$-module
\begin{equation}
	\label{induced_module}
	\IND HGW:=
	\left\{f:G\to W:
		\begin{array}{l}
			\text{(a) }f(gh)=h\inv(f(g))\quad\ \forall\,h\in H\\[1ex]
			\text{(b) }\|f\|^2:=\int_{G/H}\|f(g)\|^2\,d\dot g<\infty
		\end{array}
	\right\}
\end{equation}
where $G$ acts by $(gf)(g')=f(g\inv g')$. Here as usual the integrand depends only on $\dot g=gH$; by $d\dot g$ we mean a $G$-invariant measure on $G/H$; and members of \eqref{induced_module} are really classes of (measurable) functions modulo the relation $\|f_1-f_2\|^2=0$. For continuous $f$ this means equality, so they make a dense $G$-invariant subspace of \eqref{induced_module} called the \emph{continuous vectors}. Now Frobenius reciprocity \cite{Weil:1940} asserts:

\begin{theo}
	\label{unitary_frobenius}
	If $V$ is a unitary $G$-module and $W$ is a unitary $H$-module\textup, both finite-dimensional\textup, then we have an isomorphism of vector spaces
	\begin{equation}\
	  \label{unitary_frobenius_equality}
	  \Hom_G(V,\IND HGW)=\Hom_H(\Res^G_HV,W).
	\end{equation}
\end{theo}

\begin{proof}
	The required maps $t:A\mapsto a$ from left to right and $t\inv:a\mapsto A$ from right to left can both be defined by the commutativity (for all $g$) of the diagram
	\begin{equation}
		\label{frob_diagram}
		\begin{tikzcd}[baseline=-2pt,row sep=huge,column sep=huge,every label/.append
	  style={font=\small}]
	        V
			\ar[d,swap,"\textrm{action of }g\inv"]
			\ar[r,"A"]
			&
			\operatorname{Ind}_H^GW
			\ar[d,"\textrm{evaluation at }g" pos=0.45]
			\\
	        \Res^G_HV
			\ar[r,"a"]
			&
			W\rlap{,}
		\end{tikzcd}
	\end{equation}
	i.e. $A(v)(g)=a(g\inv v)$. We refer to \cite[p.\,31]{Wallach:1988} or \cite[p.\,303]{Kowalski:2014} for the straightforward verification that this works, as well as the subtler fact that $A$ necessarily takes values in the subspace of continuous vectors (to which, strictly speaking, our diagram should of course restrict the domain of ``evaluation at $g$'').
\end{proof}

\begin{rema}
	\label{frob_decomposition}
	The main use of \eqref{unitary_frobenius_equality} is to provide the decomposition of an induced $G$-module into irreducibles: $\IND HGW = \smash{\bigoplus_{V\in\widehat G}V\otimes\Hom_H(V,W)}$ (Hilbert sum), where $v\otimes a$ corresponds to the function $f(g) = a(g\inv v)$.
\end{rema}

\subsection{Borel--Weil correspondence}
\label{Borel-Weil}
The relation with symplectic geometry starts with Borel and Weil \cite{Serre:1954} who attached, to the irreducible $G$-module $V$ with highest weight $\lambda$, a homogeneous Kähler manifold $X$ that we now recognize as the coadjoint orbit of $\lambda$. (Normally one regards $\lambda$ as living in $\LT^*$, the dual of the Lie algebra of a maximal torus~$T$; but for maximal $T$, $\int_T\operatorname{Ad}(t)\,dt$ is a projection $\LG\to\LT$ whose dual $\LT^*\hookrightarrow\LG^*$ identifies $\LT^*$ to $\{T$-fixed points in $\LG^*\}$.) This assignment has elementary functorial properties, emphasized by Weinstein in \cite[p.\,8]{Weinstein:1981}:
\bgroup
\setlength{\extrarowheight}{2pt}
\begin{equation}
	\label{known}
	\begin{array}{|lw{c}{1.3cm}|w{c}{3.2cm}|w{c}{3.6cm}|}
		\hline
		&
		\text{group } K &
		\begin{tabular}{c}
			irreducible \\[-2pt]
			unitary $K$-module
		\end{tabular} &
		\begin{tabular}{c}
			corresponding \\[-2pt]
			Hamiltonian $K$-space
		\end{tabular}
		\\
		\hline\hline
		\!\!\texttt{a.}&
		G &
		V &
		X
		\\
		\hline
		\!\!\texttt{b.}&
		H &
		W &
		Y
		\\
		\hline
		\!\!\texttt{c.}&
		G\times H &
		V\otimes W &
		X\times Y
		\\
		\hline
		\!\!\texttt{d.}&
		G\times H &
		\Hom(V, W) &
		X^-\times Y
		\\
		\hline
		\!\!\texttt{e.}&
		G\times G &
		\operatorname{End}(V)\ (\ni g_V) &
		X^-\times X\ (\supset g_X)
		\\
		\hline
		\!\!\texttt{f.}&
		G &
		V\rlap{$^*$} &
		X\rlap{$^-$}
		\\
		\hline
		\!\!\texttt{g.}&
		\emph{any} &
		\CC \text{ (trivial)} &
		\{0\}\rlap{.}
		\\
		\hline
	\end{array}
\end{equation}

\subsection{Further expectations}
\label{expectations}
At ({\tt e}) above is Weinstein's observation that the $G$-actions themselves relate unitary operators $g_V$ to (the Lagrangian graphs~of) symplectomorphisms $g_X$. From this he drew his principle that vectors (or complex lines: see ({\tt g})) on the left should often correspond to Lagrangian submanifolds on the right. Meanwhile \cite[p.\,498]{Kazhdan:1978}, \cite[p.\,6]{Weinstein:1981} and \cite[pp.\,516, 532]{Guillemin:1982} suggested extending \eqref{known} to more (typically \emph{reducible}) modules, in such a way that more operations become intertwined:
\begin{equation}
	\label{desired}
	\begin{array}{|lw{c}{1.3cm}|w{c}{3.2cm}|w{c}{3.6cm}|}
		\hline
		\!\!\texttt{h.}&
		H &
		\Res^G_H\rlap{V}\phantom{W} &
		\Res^G_HX
		\\
		\hline
		\!\!\texttt{i.}&
		G &
		\IND HGW &
		\IND HGY
		\\
		\hline
		\!\!\texttt{j.}&
		G &
		V\oplus V' &
		X\sqcup X'
		\\
		\hline
		\!\!\texttt{k.}&
		1 &
		V^G &
		X/\!\!/G
		\\
		\hline
		\!\!\texttt{l.}&
		1 &
		\Hom_G(V,V') &
		\Hom_G(X,X')\rlap{.}
		\\
		\hline
	\end{array}
\end{equation}
\egroup
Note that \eqref{known} and \eqref{desired} are fundamentally different: the former lists known properties of an existing bijection, whereas the latter only gives prescriptions or \emph{desiderata} for a correspondence (`quantization') which hasn't yet been defined, despite 100 years of efforts \cite{Heisenberg:1925}. Clearly, unitary Frobenius \eqref{unitary_frobenius_equality} makes symplectic Frobenius \eqref{symplectic_frobenius_equality} \emph{necessary} for the existence of such a correspondence: this was our motivation for proving it. But then the (borderline meta-mathematical) question remains whether or not \eqref{desired} can be enforced a) consistently, b) by any existing `quantization' method. We discuss two examples.

\begin{Peter-Weyl}
	\label{Peter-Weyl}
	The simplest induced module is $\IND{\{e\}}{\hspace{2.5pt}G}\CC = L^2(G)$, to which ({\tt g},\,{\tt i}) prescribe attaching $\IND{\{e\}}{\hspace{2.5pt}G}\{0\}=T^*G$. Then ({\tt l}) and special case ({\tt k}), cf.~\eqref{s-intertwiner}, prescribe that if $V$ and $X$ correspond under Borel--Weil, then also
	\begin{equation}
		\label{Hom_spaces}
		\Hom_G(V,L^2(G))
		\quad\ \text{ should correspond to }\ \quad
		\Hom_G(X,T^*G).
	\end{equation}
	Now, \emph{Frobenius} (\ref{unitary_frobenius_equality},~\ref{symplectic_frobenius_equality}) \emph{equates} the first to $\Hom_{\{e\}}(V,\CC)=V^*$ and the second to $\Hom_{\{e\}}(X,\{0\})=X^-$. Moreover both \emph{happen to inherit}, from the right action, residual $G$-actions which coincide, under these identifications, with the natural (dual) ones on $V^*$ and $X^-$. So \eqref{Hom_spaces} appears here indeed both a) consistent with and b) implemented by the existing case (\ref{known}{\tt f}) of Borel--Weil. 
\end{Peter-Weyl}

\begin{spherical-harmonics}
	\label{spherical-harmonics}
	The next simplest example is with $G=\SO3$~and $H\cong\SO2$ the stabilizer of the north pole $\bm e_3\in S^2$ (a maximal torus). As is well known since Schur \cite{Schur:1924a}, $G$ has an irreducible module $V_\ell$ in each odd dimension $2\ell+1$, which under $H$ splits as the sum $\smash{\bigoplus_{m=-\ell}^\ell V\l m}$ of $1$-dimensional weight spaces (wherein rotation by angle $\varphi$ acts by factor $\e{\i m\varphi}$). Choosing basis vectors $e\l m\in V\l m$ we see that $\Hom_H(V_\ell,\CC)$ is $1$-dimensional and spanned by the form
	\begin{equation}
		\label{a_l}
		a_\ell=(e\l0|\,\cdot\,)\,.
	\end{equation}
	So \emph{Frobenius \eqref{unitary_frobenius_equality} says} that each $V_\ell$ also has multiplicity 1 in $\IND HG\CC=L^2(S^2)$: in other words, \eqref{frob_decomposition} reads $L^2(S^2)=\bigoplus_{\ell=0}^\infty A_\ell(V_\ell)$ with, by \eqref{frob_diagram},
	\begin{equation}
		\label{A_l}
		A_\ell(v)(g\bm e_3)=(ge\l0|v).
	\end{equation}
	This of course is nothing but the theory of spherical harmonics, as elucidated in \cite[p.\,127]{Weyl:1931a}, \cite[p.\,229]{Wigner:1931}: because the Laplacian on $S^2$ commutes with rotations, the $A_\ell(V_\ell)$ must be eigenspaces of it, by Schur's lemma; so the matrix coefficients $A_\ell(e\l m)$ in \eqref{A_l} must be (multiples of) Laplace's $Y\l m$.
		
	Here is the parallel symplectic computation, where as usual $\bm l\in\RR^3$ will also denote the member of $\LG^*$ sending infinitesimal rotation $\j(\bm\alpha)=\bm{\alpha\times\cdot\,}$ to $\<\bm l,\bm\alpha\>$. To $V_\ell$ and $\IND HG\CC$ we attach the coadjoint orbit $X_\ell=\{\ell\bm u:\bm u\in S^2\}$ and the Hamiltonian $G$-space $\IND HG\,\{0\}$, which as \eqref{s-induced-from-0} will recall, is simply
	\begin{equation}
		\label{TS2}
		TS^2 =
		\bigl\{x=(\bm r, \bm p)\in\RR^6: \bm r\in S^2, \<\bm r,\bm p\>=0\bigr\}
	\end{equation}
	with $G$-action $g(x)=(g\bm r,g\bm p)$, $2$-form $d\<\bm p,d\bm r\>$ and moment map $\Phi(x) = \bm{r\times p}$. On $X_\ell$ the $H$-action has moment map the orthogonal projection to the axis $\RR\bm e_3$ (cf.~first parenthesis in §\ref{Borel-Weil}). Hence the reduced space $\Hom_H(X_\ell,\{0\})$ is just a point, as quotient \eqref{s-intertwiner} by $H=\SO2$ of a level which is the equator (or zero):
	\begin{equation}
		\label{H-level}
		(\Psi_{X_\ell^-\times\{0\}})\inv(0)
		=
		\bigl\{
			(\ell\bm u,0)
			:
			\bm u\in S^2, \<\bm e_3,\bm u\>=0
		\bigr\}.
	\end{equation}
	So \emph{Frobenius \eqref{symplectic_frobenius_equality} says} that $\Hom_G(X_\ell,TS^2)$ is also a point --- which of course we can also exhibit directly, as quotient \eqref{s-intertwiner} by $G=\SO3$ of the level
	\begin{equation}
		\label{G-level}
		(\Phi_{X_\ell^-\times TS^2})\inv(0)
		=
		\bigl\{
			(\ell\bm u,\bm r,\ell\bm s)
			:
			(\bm r\ \bm s\ \bm u)\in\SO3
		\bigr\}.
	\end{equation}
	This is the zero section of $TS^2$ when $\ell=0$, and else manifestly also a single $\SO3$-orbit. Note that \eqref{H-level} and \eqref{G-level} are Lagrangian, being both a zero level (coisotropic) and a group orbit in it (isotropic).
\end{spherical-harmonics}

\begin{scho}
	\label{scholium}
	The upshot of this example is that prescription (\ref{desired}{\tt l}) \emph{must attach} $\CC a_\ell$ and $\CC A_\ell$ in (\ref{a_l}--\ref{A_l}) to the two \emph{points} \eqref{H-level}$/H$ and \eqref{G-level}$/G$. This, to address question a) of §\ref{expectations}, is thankfully consistent with Frobenius reciprocity on both sides, and with the idea (\ref{known}{\tt g}) that points should match complex lines and thus `predict multiplicity $1$'.
	
	But why e.g.~\emph{this} complex line, $\CC A_\ell$? Unlike the $\Hom$ spaces \eqref{Hom_spaces}, the point \eqref{G-level}$/G$ has not enough structure to encode that. Instead, one might argue, it remembers \emph{being} the Lagrangian $G$-orbit \eqref{G-level} (and indeed \emph{the unique} $G$-invariant Lagrangian submanifold of $X_\ell^-\times TS^2$: these all lie in the zero level, by virtue of $\<\Phi(x),\LG\>=\<\Phi(x),[\LG,\LG]\>=\omega(\LG(x),\LG(x))=0$); and then it invokes Weinstein's principle to conclude that it should attach itself to \emph{the unique} $G$-invariant complex line in $\Hom(V_\ell,L^2(S^2))$: namely $\CC A_\ell$.
	
	While this sounds compelling, we should note that it is not the same thing as a construction --- such as we have in e.g.~\eqref{Hom_spaces}. In fact, to now address question b), we confess to not knowing any quantization method that would produce the intertwining operator \eqref{A_l} out of the Lagrangian submanifold \eqref{G-level}.
\end{scho}

\subsection{Beyond compact $\mathbf G$}
\label{noncompact}
Correspondences like (\ref{known}{\tt a}) hold for many more Lie groups, perhaps most famously the \emph{exponential groups} (i.e.~$\exp:\LG\to G$ is a diffeomorphism) after Kirillov and Bernat, whose map we can describe as follows \cite[\nolinebreak p.\,30]{Ziegler:1996}: by an easy recursion, every coadjoint orbit $X$ is $\IND{G'}G\{x'\}$ for some singleton orbit of some exponential subgroup $G'$; then as (\ref{desired}{\tt i}) suggests, we map $X$ to $V=\IND{G'}G\CC_{x'}$, where $G'$ acts in $\CC_{x'}$ by the character $\exp(Z)\mapsto\e{\i\<x',Z\>}$.

We refer to \cite{Kirillov:1962,Vergne:1970c} for the fact that this yields a bijection $\LG^*/G\to\widehat G$ from coadjoint orbits onto equivalence classes of unitary $G$-modules. Then again \eqref{desired} suggests that our (isomorphic) reduced spaces in \eqref{symplectic_frobenius_equality},
\begin{equation}
	\label{Hom_spaces_again}
	\text{(a) }\Hom_G(X,\IND HGY)
	\qquad\text{and}\qquad
	\text{(b) }\Hom_H(\Res^G_HX,Y),
\end{equation}
will reflect multiplicity (and Frobenius reciprocity) in decompositions of induced and restricted modules into irreducibles. A difficulty here is that generally these are no longer direct sums \eqref{frob_decomposition} but direct integrals, so unitary Frobenius reciprocity as a `pointwise' equality \eqref{unitary_frobenius_equality} is lost without agreed-upon replacement: see e.g. \cite[\nolinebreak pp.\,968, 996, 1370]{Fell:1988a}. But in a sense (\ref{Hom_spaces_again}a) $=$ (\ref{Hom_spaces_again}b) can fill this role, as results do exist relating these spaces to multiplicities when $X$ and $Y$ correspond (under Kirillov--Bernat) to modules $V$ and $W$:

\begin{exes}
	\ 
	\begin{compactenum}[a)]
		\item In \cite{Vergne:1970c}, Michèle Vergne considers pairs $(H,Y)$ of a connected subgroup $H\subset G$ and a singleton coadjoint orbit $Y=\{y\}\subset\LH^*$, such that $\IND HG\CC_y$ is a \emph{finite sum} of irreducibles $V$; and she shows, in perhaps the earliest instance of `quantization commutes with reduction', that the $V$ that occur do so with multiplicity equal to \emph{the cardinality of} (\ref{Hom_spaces_again}b). This includes in particular Pukánszky's irreducibility criterion, satisfied by our pairs $(G',\{x'\})$ above and first understood symplectically in \cite{Duval:1992}.
		\item In \cite{Corwin:1988,Fujiwara:1991}, Corwin, Greenleaf and Fujiwara show that the cardinality of (\ref{Hom_spaces_again}b) can always be taken as multiplicity function in the direct integral decompositions of $\IND HGW$ and $\Res^G_HV$, \emph{both}. They call this coincidence `a~form of Frobenius reciprocity, which we hope to explore in the future'. 
		\item
		Beyond exponential groups, recent papers (e.g.~\cite{Halima:2018,Kobayashi:2018,Hochs:2019}) discuss this cardinality, under the name \emph{Corwin--Greenleaf multiplicity function of $(X,Y)$}, in more cases where one has an `orbit method' correspondence.
	\end{compactenum}
\end{exes}

\subsection{Role of prequantization}
Finally, one might ask why we take the trouble to treat not only Hamiltonian $G$-spaces $X$ but also prequantum $G$-spaces $\tilde X$. The reason is that the latter are known (in many cases) or expected (in general) to much better reflect the theory of unitary $G$-modules. This manifests itself in a number of ways, of which we will mention just three.

•\ First, even for $G$ compact, the correspondence $V\mapsto X$ (\ref{known}{\tt a}) is \emph{not~onto} all coadjoint orbits: instead, its range is precisely all prequantizable $X$. (This means: $[\omega]\in H^2(X,\ZZ)$ \emph{and} $G$ lifts to act on $\tilde X$, with moment map \eqref{prequantum_moment}. Both hold by Borel and Weil's original identification of $X$ as the $G$-orbit in $\mathbf P(V)$ of a projectivized highest weight vector: for we may then take as $\tilde X$ the tautological circle bundle above it, on which $G$ does act.) \emph{Likewise}, in (\ref{desired}{\tt l}) one expects only `integral' reduced spaces $\Hom_G(X,X')$ to yield nonzero $\Hom_G(V,V')$: this is part of the Guillemin--Sternberg conjecture \cite{Guillemin:1982}, of which many variants have since been proved --- see \cite{Sjamaar:1996,Guillemin:2002,Vergne:2002}.

\begin{exem}
	In (\ref{frob_decomposition},~\ref{Peter-Weyl}), $L^2(S^1)=\bigoplus_{n\in\ZZ}\CC(\e{\i\varphi}\mapsto\e{\i n\varphi})$. In other words, the subscripts in trigonometric series are \emph{integer} coadjoint orbits $\{n\}$ of the circle group.
\end{exem}

•\ Secondly, when we move beyond compact or exponential to e.g.~solvable~$G$,  prequantized rather than bare coadjoint orbits parametrize irreducible $V$ \cite{Auslander:1967}, and prequantization ceases to be always unique. Thus $V\mapsto X$ is \emph{not one-to-one}: instead, $V\mapsto\tilde X$ is. Therefore, no prequantization-unaware theory can possibly hope to fully describe multiplicities and their reciprocity.

\begin{exem}
	Section 9 of \cite{Ratiu:2022} exhibits a Lie group $G$ with an induced coadjoint orbit $\IND HGY\cong TS^1$ having infinitely many prequantizations $\IND HG{\tilde Y_\lambda}$ ($\lambda\in\RR/\ZZ$) associated to irreducible unitary $G$-modules $V_\lambda$; and a Hamiltonian $G$-space $X'\cong\RR^2$ (its universal covering) having a unique prequantization $\tilde X'$ associated to a \emph{reducible} unitary $G$-module $V'$. Then \emph{Frobenius \eqref{prequantum_frobenius_equality} computes}
	\begin{equation}
		\Hom_G(\tilde X',\IND HG{\tilde Y_\lambda}) =\TT
		\quad
		\text{for all $\lambda$},
	\end{equation}
	which correctly predicts that $V'=\int^\oplus V_\lambda\,d\lambda$ with uniform multiplicity $1$. (Here `associated' is in the precise sense of \cite{Auslander:1967}, as $G$ is connected type I solvable, and $X'$ is also a coadjoint orbit of a connected type I solvable overgroup $G'\supset G$.)
\end{exem}

•\ Third, prequantization lets us recast Weinstein's principle (§\ref{expectations}) in terms of Legendrian submanifolds, i.e.~$\smash{\frac12}\dim(X)$-dimensional 
submanifolds of $\tilde X$ on which $\varpi$ vanishes \cite{Arnold:2000}: arguably \emph{those} are what best corresponds to individual vectors in (\ref{known}{\tt a}) --- perhaps only when they project bijectively onto Lagrangian submanifolds of $X$; cf.~the discussion of \emph{Planck lifting} in \cite[18.111, 19.25]{Souriau:1970}. One way in which this immediately improves matters is:

\begin{prop}
	\label{bijection}
	Let $\tilde X$ be a prequantum $G$-space with reduced space $\tilde X/\!\!/G$\textup, and assume that the $G$-action on $\level=\Phi\inv(0)$ is free and proper \textup{(\ref{prequantum_reduction_diagram},~\ref{two_prequantum_settings})}. Then $M\mapsto L=\pi\inv(M)$ defines a bijection between Legendrian submanifolds $M$ of $\tilde X/\!\!/G$\textup, and $G$-invariant Legendrian submanifolds $L$ of $\tilde X$.
\end{prop}

\begin{proof}[Sketch of proof]
	Guillemin and Sternberg \cite[Thm 2.6]{Guillemin:1982} prove the same thing with ($\tilde X$, prequantum, Legendrian) replaced by ($X$, Hamiltonian, Lagrangian), under the pesky hypothesis $\LG=[\LG,\LG]$ needed to show, just as we did in \eqref{scholium}, that $G$-invariant Lagrangian submanifolds lie in $\Phi\inv(0)$. (We say \emph{needed}, because parallels to the equator fail this in our $\SO2$-space $S^2$: hence why, as alert readers noticed, \eqref{scholium} did not call \eqref{H-level} `the unique' $H$-invariant Lagrangian.) Now in our new setting, \eqref{prequantum_moment} alone forces $G$-invariant Legendrian submanifolds to be in $\Phi\inv(0)$, and the rest of \cite{Guillemin:1982}'s proof still works \emph{mutatis mutandis}.
\end{proof}

\begin{exem}
	Let us now prequantize the computation (\ref{TS2}--\ref{G-level}), keeping its notation. Following \cite{Serre:1954}, the sphere $X_1=S^2$ is prequantized by the $G$-orbit of the (highest weight) vector $\smash{\frac1{\sqrt2}(\bm e_1-\i\bm e_2)}$ in the $G$-module $V_1=\CC^3$:
	\begin{equation}
		\label{prequantized_sphere}
		\tilde X_1=
		\Bigl\{
			\xi=
			\tfrac1{\sqrt2}(\bm u_1-\i\bm u_2)
			:
			(\bm u_1\,\bm u_2\,\bm u_3)\in\SO3
		\Bigr\}
	\end{equation}
	with contact $1$-form $\varpi\subtilde X1(\delta\xi) = \frac1\i(\xi|\delta\xi) = \<\bm u_2,\delta\bm u_1\>$, Reeb $\TT$-action $z(\xi)=z\xi$, and projection onto $X_1$ given by the moment map \eqref{prequantum_moment}: $\Phi\subtilde X1(\xi) = \bm u_1\bm\times\bm u_2=\bm u_3$. We emphasize that $\j(\bm u_3)\xi=\i\xi$, so that $\e{\i\thetaup}\xi$ is also the rotated vector $\e{\j(\thetaup\bm u_3)}\xi$. Next $\tilde X_\ell$ and $\varpi\subtilde X\ell$ are the image and push-forward of $\tilde X_1$ and $\ell\varpi\subtilde X1$ under the~map $\xi\mapsto\xi^\ell$ of $\CC^3$ to $\operatorname{Sym}^\ell(\CC^3)$ (`fusion method' of \cite{Souriau:1967,Souriau:1970}: that map's fibers are the orbits of the Reeb action of the $\ell$th roots of $1$, and $\tilde X_\ell$ is also known as the lens space $L(2\ell,1)$ \cite{Bredon:1993}). Finally $\{0\}$ in (\ref{known}{\tt g}) is prequantized by $\{\tilde0\}=\TT$ with $1$-form $dz/\i z$ and trivial $H$-action, and $TS^2$ trivially by
	\begin{equation}
		\tilde{TS}{}^2 = TS^2\times\TT
		\qquad\text{with}\qquad
		\varpi_{\tilde{TS}{}^2} = \<\bm p,d\bm r\>+\frac{dz}{\i z}.
	\end{equation}
	Now the relevant zero levels inside prequantum products \eqref{p-intertwiner} are the preimages there of our Lagrangian submanifolds \eqref{H-level} and \eqref{G-level}: using ($H$- resp.~$G$-) equivariance, one readily finds
	\begin{align}
		(\Psi_{\tilde X_\ell^-\boxtimes\{\tilde0\}})\inv(0)&=
		\Bigl\{
			\Bigl(
				\frac{\bm e_3-\i\bm s}{\mbox{\small$\sqrt2$}}
			\Bigr)^\ell
			\boxtimes z
			:
			\bm s\in S^2, \<\bm e_3,\bm s\>=0,
			z\in\TT
		\Bigr\},
		\\[2ex]
		(\Phi_{\tilde X_\ell^-\boxtimes\tilde{TS}{}^2})\inv(0)&=
		\Bigl\{
			\Bigl(\frac{\bm r-\i\bm s}{\mbox{\small$\sqrt2$}}\Bigr)^\ell
			\boxtimes(\bm r,\ell\bm s,z)
			:
			(\bm r\ \bm s\ \bm u)\in\SO3,
			z\in\TT
		\Bigr\},
	\end{align}
	both of which visibly split as the product of $\TT$ with any one ($H$- resp.~$G$-) orbit, those being simply transitive for $\ell\ne0$ and Legendrian by dimension and \eqref{prequantum_moment}. Hence we find, \emph{as predicted by Frobenius \eqref{prequantum_frobenius_equality}}, that the two quotients $\Hom_H(\tilde X_\ell,\{\tilde0\})$ and $\Hom_G(\tilde X_\ell,\tilde{TS}{}^2)$ are equal. More precisely they are equal to $\TT$, which indicates multiplicity $1$; and since the connected Legendrian submanifolds of $\TT$ are just its points, this also illustrates \eqref{bijection}, twice. Our remarks in \eqref{scholium} still hold \emph{mutatis mutandis}.
\end{exem}

\part{Reduced forms}

The purpose of this Part is to give new sufficient conditions, improving on the state-of-the-art (\ref{two_settings}/\ref{two_prequantum_settings}), for the existence of reduced forms on reduced spaces. The main question: whether or when a differential form is pulled back from a quotient, is of course an old one, with two main strands.

(a) \emph{Integral invariants}. Here the quotient is the leaf space of a foliation, and a form that ``should'' descend is called an \emph{integral invariant}. The name comes from Poincaré, who meant invariance along flows of actual integrals of forms over cycles or chains --- e.g.~Helmholtz's circulation of velocity, or flux of vorticity, in fluid flow: see \cite[§§234--235]{Poincare:1899} and \cite[§60\emph{a\textup,n\textup,r}]{Thomson:1869}. Soon after, É.~Cartan refocused attention on the integrands, calling a form $\form$ integral \cite{Cartan:1902} or invariant \cite[§31]{Cartan:1922} if ``it can be expressed by means of first integrals and their differentials alone'', and proving that such is the case iff all vectors $v$ tangent to the leaves satisfy $\form$'s \emph{characteristic system}: $\i_v\form=\i_vd\form=0$ \cite[§78]{Cartan:1922}. (Cartan didn't speak of interior products, nor leaf spaces: that was done by Ehresmann \cite{Ehresmann:1946} and Reeb \cite[§4, Prop.~2 and 5]{Reeb:1952a}.) Later Souriau popularized the process of quotienting by (characteristic) foliations, by making it a cornerstone in his symplectic treatment of mechanics and quantization \cite{Souriau:1967,Souriau:1970}. He also traced the idea back to Lagrange, in a theorem that saw $t, \Delta t, \delta t$ ``disappear'' from a certain $2$-form: see \cite[p.\,329]{Lagrange:1811}, \cite[§8]{Cayley:1858}, \cite[§146]{Whittaker:1904a}, \cite[p.\,47]{Souriau:1986}.

(b) \emph{Basic forms}. Here the quotient is the orbit space of the action of a Lie group $G$, and a form $\form$ that ``should'' descend is called \emph{basic}. The name came in stages from H.~Cartan \cite[§4]{Cartan:1950}, Reeb \cite[§4]{Reeb:1952a}, and finally Koszul \cite[§1]{Koszul:1953}, for whom basic means \emph{$G$-invariant and horizontal}: $g^*\form=\form$ for all $g\in G$, and $\i_{Z(\cdot)}\form=0$ for all $Z\in\LG$. Special cases already appear, without the name, in \cite[13.7--8]{Chevalley:1948} and even \cite[§5]{Cartan:1929}.

Both (a) and (b) say ``should'', because traditionally, pulling forms back from a quotient only makes sense if that quotient is a \emph{manifold}. To ensure this, one has little choice but to make standard assumptions such as: in (a), the foliation admits \emph{transverse sections} \cite[9.2.9]{Bourbaki:1967}; resp.~in (b), the $G$-action is \emph{free and proper} \cite[\nolinebreak III.1.5, Prop.~10]{Bourbaki:1972}. Then indeed, the quotient is a manifold, and pull-back gives an isomorphism between the complex of differential forms on the quotient and the integral invariants (resp.~the basic forms): see \cite[5.21]{Souriau:1970} (resp.~\cite[§1]{Koszul:1953} for $G$ compact and \cite[p.\,185]{Guillemin:2002} in general).

Returning to symplectic reduction \eqref{reduction_diagram}, \cite{Marsden:2007} recounts that Smale \cite{Smale:1970} inspired \cite{Meyer:1973,Marsden:1974} to divide $\Phi\inv(0)$ by the $G$-action: then, assuming it \emph{free and proper}, (b) recovers precisely the Marsden--Weinstein result \eqref{two_settings}. To go beyond that (hence necessarily outside the realm of manifolds), perhaps the most successful foray is the theory of Sjamaar, Lerman and Bates \cite{Sjamaar:1991,Bates:1997} who assume a \emph{proper} but not free $G$-action on $X$. Then they show that $X/\!\!/G$ is a (Whitney) `stratified symplectic space', i.e.\,(among other things) a disjoint union of manifolds, each carrying a symplectic form, whose associated Poisson brackets fit together by all being induced by the global Poisson structure of~\cite{Arms:1991}.

Now diffeology offers another possible avenue, because the subset $C=\Phi\inv(0)$ and quotient $C/G=X/\!\!/G$ in \eqref{reduction_diagram} are canonically diffeological: so, regardless of whether either is a manifold, we can talk about $j^*\omega$ and ask if it is pulled back from a (necessarily unique, global, diffeological) form on the quotient. Souriau \cite{Souriau:1985a} called this being an \emph{integral invariant of the subduction $\pi$}, and gave for it a fundamental criterion, recalled below in \eqref{desired_equal_pull_backs}.

Past works using this criterion include \cite{Hector:2011} which generalized Souriau's result for (a) to regular foliations \emph{without} transverse sections; \cite{Watts:2012} which generalized Koszul's result for (b) to \emph{non-free} compact Lie group actions; \cite{Karshon:2016} which extended it to \emph{non-free} proper Lie group actions; \cite{Watts:2022} which extended it further to orbit spaces of proper Lie groupoids; and \cite{Miyamoto:2023} which extended \cite{Hector:2011} to certain singular foliations. \emph{In all these cases the space being quotiented is a manifold\textup, whereas for us $\Phi\inv(0)$ need not be.}

In the following sections, we use Souriau's criterion to solve the existence problem for reduced forms (\ref{reduced_2-form}/\ref{reduced_1-form}) when the $G$-action on $\Phi\inv(0)$ is either strict \eqref{reduced_forms_for_strict_actions} or locally free \eqref{reduced_forms_at_regular_values}, or the $G$-action on $X$ is proper \eqref{reduced_forms_for_proper_actions}. The results are~as described in the Introduction; in particular, in the proper case we obtain a global reduced $2$-form that induces the Sjamaar--Lerman--Bates $2$-forms on strata.

\section{When the $G$-action on \texorpdfstring{$\Phi\inv(0)$}{Φ-1(0)} is strict}\label{s:strict}

Whenever a group $G$ acts on a set $X$, we can consider the resulting \emph{Bourbaki~map} $\thetaup: (g,x)\mapsto(x,g(x))$ \cite[III.4.1]{Bourbaki:1960} and factor it as in the following diagram, where $\mathrm{s}$ is the quotient map by the equivalence relation `equal value under $\thetaup$', $\dot\thetaup$ is the resulting bijection of that quotient with $\Im(\thetaup)$, and $\mathrm{i}$ is the injection of that image into $X\times X$:
\begin{equation}
	\label{factorization}
	\begin{tikzcd}[row sep=huge,column sep=huge,every label/.append
  style={font=\small}]
        G\times X \ar[d,swap,"\mathrm{s}"]\ar[r,"\thetaup"] & X\times X \\
        (G\times X)/{\sim}\ar[r,"\dot\thetaup"] & \thetaup(G\times X)\ar[u,swap,hook,"\mathrm{i}"],
	\end{tikzcd}
	\qquad
	\thetaup = \mathrm{i}\circ\dot\thetaup\circ\mathrm{s}.
\end{equation}
If $G$ and $X$ are diffeological and the action is smooth, then all nodes here have natural diffeologies (product, quotient, subset), $\mathrm{i}$ is an induction, $\mathrm{s}$ a subduction, and $\dot\thetaup$ is smooth \cite[1.34, 1.51]{Iglesias-Zemmour:2013}. Following \cite{Souriau:1985a,Iglesias-Zemmour:2013} we say that $\thetaup$ is \emph{strict} if $\dot\thetaup$ is a diffeomorphism, i.e., $\dot\thetaup$ \emph{and} $\dot\thetaup\inv$ are smooth.

\begin{defi}
	\label{strict}
	We will call the $G$-action \emph{strict} if $\thetaup$ is a strict map.
\end{defi}
\begin{exes}
	\label{principal}
	A \emph{free} action (wherein $\thetaup$ is injective) is strict iff it is \emph{principal} in the sense of \cite[8.11]{Iglesias-Zemmour:2013}, i.e., $\thetaup$ is an induction. (Indeed, a strict injection is the same thing as an induction \cite[1.20]{Souriau:1985a}.) Thus we have for instance:
	\begin{compactenum}[a)]
		\item If $G$ is a Lie group and $H$ an arbitrary subgroup (hence canonically also a Lie group \cite[III.4.5]{Bourbaki:1972}), then the left and `right' actions of $H$ on $G$, $h(g)=hg$ and $h(g)=gh\inv$, are principal and hence strict \cite[8.15]{Iglesias-Zemmour:2013}.
		\item More generally, any free smooth action of a Lie group on a manifold is principal and hence strict, according to \cite[3.9.3]{Iglesias:1985}.
	\end{compactenum}
	A \emph{transitive} action (wherein $\thetaup$ is surjective) is strict iff $\thetaup$ is a subduction: see again \cite[1.20]{Souriau:1985a}. For instance:
	\begin{compactenum}
		\item[c)] Any transitive smooth action of a Lie group on a manifold is strict. (For $\thetaup$ is then a surjective submersion, hence a subduction \cite[1.13]{Souriau:1985a}.)
		\item[d)] Actions that are neither free nor transitive can easily be non-strict: thus already the standard action of $\SO2$ on $\RR^2$. (Indeed, the reader will check without trouble that \eqref{smooth_division} below fails for the $1$-plots $P(u) = (0, \smash{\e{-1/u^2}})$ and $Q(u) = (0, \operatorname{sign}(u)\smash{\e{-1/u^2}})$, both understood to be $(0,0)$ at $u = 0$.)
	\end{compactenum}
\end{exes}

Now return to the setting of (\ref{reduced_space}/\ref{prequantum_reduced_space}): the $G$-action of interest is on a moment level $\level=\Phi\inv(0)$ (not necessarily a manifold), and the relevant Bourbaki map writes $\thetaup: G\times \level\to \level\times\level$.

\begin{theo}
	\label{reduced_forms_for_strict_actions}
	In \textup{(\ref{reduced_space}/\ref{prequantum_reduced_space}),} suppose the $G$-action on $\level=\Phi\inv(0)$ is strict. Then $X/\!\!/G$ carries a reduced $2$-form \textup(resp.~$\tilde X/\!\!/G$ carries a reduced $1$-form\textup).
\end{theo}

\begin{proof}
	Let $j$ and $\pi$ be as in \eqref{reduction_diagram}. The criterion of \cite[2.5c]{Souriau:1985a} or \cite[6.38]{Iglesias-Zemmour:2013} says: a reduced 2-form will exist iff, given any two plots $P:U\to \level$ and $Q:U\to \level$ such that $\pi\circ P=\pi\circ Q$, we have
	\begin{equation}
		\label{desired_equal_pull_backs}
		P^*j^*\omega = Q^*j^*\omega
	\end{equation}
	(an equality of ordinary $2$-forms on an open set $U$ in some $\RR^n$). To prove \eqref{desired_equal_pull_backs}, fix $P$, $Q$, $u_0\in U$ and vectors $\delta u_0,\delta'u_0\in T_{u_0}U$. Note that $\pi\circ P=\pi\circ Q$ says that~$P\times Q$ is a \emph{plot of $\level\times\level$ taking values in} $\Gamma:=\thetaup(G\times\level)$. So strictness of $\thetaup$, as expressed in \cite[1.54]{Iglesias-Zemmour:2013}, means: for any such $P\times Q$, there are an open neigh\-bor\-hood $V$ of $u_0$, and a plot $R\times S: V\to G\times\level$, such that $\thetaup\circ(R\times S)=(P\times Q)_{|V}$. Now this says that $S=P_{|V}$ and $R$ is a smooth map $V\to G$ satisfying
	\begin{equation}
		\label{smooth_division}
		Q(u)=R(u)(P(u))\rlap{\qquad\qquad$\forall\,u\in V$.}
	\end{equation}
	Also, by definition of the subset diffeology of $\level$ in $X$, $j\circ P_{|V}$ and $j\circ Q_{|V}$ are smooth maps $V\to X$, taking values in $\level$. So all the successive variables
	\begin{equation}
		g=R(u),\qquad
		x=j(P(u)),\qquad
		y=g(x)=j(Q(u)),\qquad
		\mu = \Phi(x)
	\end{equation}
	are ordinary smooth functions of $u\in V$, valued respectively in $G$, $X$, $X$, and~$\LG^*$. Hence, derivatives of these functions (and left translation by $g\inv$) will map each $\delta u\in T_uV$ to vectors $\delta g$, $\delta x$, $\delta y$, $\delta\mu$ and an element $Z:=g\inv\delta g\in\LG$. Computing again as in (\ref{action_map_derivative}, \ref{pull_back_computation}) (but with $\delta g = gZ$), we get $\delta y =g_*(\delta x + Z_X(x))$ and
	\begin{equation}
		\label{other_pull_back_computation}
		\begin{aligned}
			(Q^*j^*\omega)(\delta u,\delta'u)
			&=\omega(\delta y,\delta'y)\\
			&=\omega(\delta x + Z_X(x),\delta'x+Z'_X(x))\\
			&=\omega(\delta x,\delta'x) 
			 +\<\delta\mu,Z'\> 
			 -\<\delta'\mu,Z\>
			 + \<\mu,[Z',Z]\>\\
			&=\omega(\delta x,\delta'x)\\
			&=(P^*j^*\omega)(\delta u,\delta'u).
		\end{aligned}
	\end{equation}
	Here the third equality is because $\Phi$ is an (equivariant) moment map, and the fourth is because $P(U)\subset\Phi\inv(0)$. In particular \eqref{other_pull_back_computation} holds for our pair $\delta u_0,\delta'u_0$, which was arbitrary in $TU$. So \eqref{desired_equal_pull_backs} is proved, and we have our reduced $2$-form. To get the reduced $1$-form on $\tilde X/\!\!/G$, argue \emph{mutatis mutandis} the same, with now (in the notation of §§\ref{prequantum} and \ref{prequantum_Frobenius})
	\begin{align}
		(Q^*j^*\varpi)(\delta u)
		&=\varpi(\delta\tilde y)\notag\\
		&=\varpi(\delta\tilde x + Z_{\tilde X}(\tilde x))\notag\\
		&=\varpi(\delta\tilde x) + \<\Phi(\tilde x),Z\>\\
		&=\varpi(\delta\tilde x)\notag\\
		&=(P^*j^*\varpi)(\delta u).\notag\qedhere
	\end{align}
\end{proof}

\section{When the $G$-action on \texorpdfstring{$\Phi\inv(0)$}{Φ-1(0)} is locally free}\label{s:regular}

\begin{theo}
	\label{reduced_forms_at_regular_values}
	In \textup{(\ref{reduced_space}/\ref{prequantum_reduced_space}),} suppose that $G$ is connected and the $G$-action on $\level=\Phi\inv(0)$ is locally free. Then $X/\!\!/G$ carries a reduced $2$-form \textup(resp.~$\tilde X/\!\!/G$ carries a reduced $1$-form\textup).
\end{theo}

\begin{proof}
	In \eqref{reduced_space}, $G$ acting locally freely on $\level$ means $\LG_x=\{0\}$ for all $x\in\level$. Then (\ref{cardinal}b) shows that $0$ is a regular value of $\Phi$, so $\level$ is a manifold, and the $G$\nobreakdash-orbits in $\level$ have tangent spaces $\LG(x)$ of constant dimension $\dim(G)$. Hence they are the leaves of a \emph{foliation} $\mathscr F$ of $\level$: see \cite[9.3.3(iv)]{Bourbaki:1967}, or the direct construction of foliating charts in \cite[p.\,13]{Ban:2006}. Furthermore, (\ref{cardinal}a) and $\Ker(D\Phi(x))=T_x\level$ show that $\omega_{|\level}$ is \emph{basic}, i.e., $G$-invariant with $\LG(x)\subset\Ker(\omega_{|\level})$ for all $x\in \level$. So \cite[Thm 3.5]{Hector:2011} applies and shows that $\omega_{|\level}\ (=j^*\omega)$ is the pull-back of a diffeological $2$-form on $\level/\mathscr F=\level/G=X/\!\!/G$, as desired.
	
	In \eqref{prequantum_reduced_space}, $G$ acting locally freely on $\level$ means $\LG_{\tilde x}=\{0\}$ for all $\tilde x\in\level$. Now, writing $x$ for the circle $\TT(\tilde x)$ (§\ref{prequantum}), one knows generally that $G_x$ acts in $\TT(\tilde x)$ via a character $\chi:G_x\to\TT$ with differential $D\chi(e) = \i\Phi(\tilde x)_{|\LG_x}$ (see e.g.~\cite[4.3d]{Souriau:1988}). Thus we have $\LG_{\tilde x}=\Ker D\chi(e)\subset\LG_x$, which when $\tilde x\in\level$ implies $\LG_x = \LG_{\tilde x} = \{0\}$ and therefore
	$
		\Im(D\Phi(\tilde x))=
		\Im(D\underline\Phi(x))=
		\ann(\LG_x)=
		\LG^*;
	$
	here we have applied (\ref{cardinal}b) to the moment map $\underline\Phi$ defined after \eqref{prequantum_moment}. So again $0$ is a regular value of $\Phi$, $\level$ is a manifold, and the $G$-orbits in $\level$ are the leaves of a foliation. Moreover \eqref{prequantum_moment} shows at once that $\varpi_{|\level}$ is basic, i.e., $G$-invariant with $\LG(\tilde x)\subset\Ker(\varpi_{|\level})$ for all $\tilde x\in \level$. So again \cite{Hector:2011} applies and gives the desired conclusion.
\end{proof}

\section{When the $G$-action on $X$ is proper}\label{s:proper}

Recall that the action of a Lie group on a manifold is \emph{proper} if $\thetaup$ in \eqref{factorization} is a proper map, i.e., compact subsets have compact preimages.  

\begin{theo}
	\label{reduced_forms_for_proper_actions}
	In \textup{(\ref{reduced_space}/\ref{prequantum_reduced_space}),} suppose the $G$-action on $X$ \textup(resp.~$\tilde X$\textup) is proper. Then $X/\!\!/G$ carries a reduced $2$-form \textup(resp.~$\tilde X/\!\!/G$ carries a reduced $1$-form\textup).
\end{theo}

\begin{proof}[Proof \textup(symplectic case\textup)]
	Let $j$, $\pi$, and $\Phi\inv(0)=\level$ be as in \eqref{reduction_diagram}. We must prove \eqref{desired_equal_pull_backs} for any two plots $P,Q:U\to\level$ with $\pi\circ P=\pi\circ Q$. To this end, we consider the following commutative diagram (after \cite[3.41]{Watts:2012}):
	\begin{equation}
		\label{strata}
	  	\begin{tikzcd}[row sep=large,column sep=large,every label/.append style={font=\scriptsize}]
	          & U \ar[r,"P"] 
			  & \level \ar[r,hook,"j"] 
			  & X \\
	            V_t\ar[r,hook]
			  & U_t\ar[u,hook]\ar[r,"P_{|U_t}"]
			  & \level_t\ar[u,hook]\ar[r,hook]
			  & X_t\ar[u,hook]\rlap{.}
	  	\end{tikzcd}
	\end{equation}
	Here $t$ is an \emph{orbit type}, i.e.~a conjugacy class of closed subgroups of $G$, and $X_t:=\{x\in\nolinebreak X:G_x\in t\}$ is the resulting \emph{orbit type piece}; $\level_t$ and $U_t$ are its preimages under $j$ and $P$, i.e.~$\level_t = \level\cap X_t$ and $U_t = P\inv(\level_t)$; and we define $V_t=U_t\cap\interior(\closure(U_t))$, closure and interior taken in the euclidean topology of~$U$. Note that $\pi\circ P = \pi\circ Q$ implies that $U_t$ is also $Q\inv(\level_t)$.
	
	One knows that the orbit type \emph{partitions} $X$ into locally finitely many such $X_t$, all of which are $G$-invariant embedded submanifolds (hence locally closed \cite{Bourbaki:1967}): see \cite[IX.9.4, Thm 2]{Bourbaki:1982} and \cite[3.5.2]{Wall:2016}. (Wall actually proves local finiteness of a finer partition; we follow \cite[5.1.8]{Bourbaki:1967} in allowing (sub)manifolds with connected components of different dimensions.) In particular, by shrinking $U$ to the preimage of some open neighborhood of $j(P(u_0))$ if necessary, we can assume without loss that only finitely many $U_t$ are nonempty.
		
	Moreover, by \cite[Thm 2.1]{Sjamaar:1991} for compact $G$ and \cite{Bates:1997} or \cite[Thm 8.1.1]{Ortega:2004a} in general, the $\level_t$ are also $G$-invariant embedded submanifolds of $X$, and the resulting $\level_t/G$ are \emph{symplectic manifolds}, carrying $2$-forms $\omega_t$ characterized by
	\begin{equation}
		\label{reduced_2_forms}
		(j_{|\level_t})^*\omega = (\pi_{|\level_t})^*\omega_t^{\phantom*}.
	\end{equation}
	Now if the $U_t$ were all open, we would be done. Indeed, $P_{|U_t}$ and $Q_{|U_t}$ would then be \emph{plots} of $\level$ with values in $\level_t$, i.e., plots of the subset diffeology of $\level_t$ (in $\level$ or equivalently in $X$ \cite[1.35]{Iglesias-Zemmour:2013}), i.e., ordinary smooth maps from $U_t$ to the manifold $\level_t$ \cite[4.1, 4.3C]{Iglesias-Zemmour:2013}. Pulling \eqref{reduced_2_forms} back by both, we would get, on the left, $((j\circ P)_{|U_t})^*\omega$ and $((j\circ Q)_{|U_t})^*\omega$, and on the right, $((\pi\circ P)_{|U_t})^*\omega_t^{\phantom*}$ and $((\pi\circ Q)_{|U_t})^*\omega_t^{\phantom*}$, which are equal since $\pi\circ P=\pi\circ Q$. So we would conclude equality of the left-hand sides, i.e.,
	\begin{equation}
		\label{equal_on_subsets}
		(j\circ P)^*\omega
		\quad\text{and}\quad
		(j\circ Q)^*\omega
		\quad\text{coincide on every $U_t$},
	\end{equation}
	whence the desired equality \eqref{desired_equal_pull_backs} on the union $U$ of the $U_t$.
	
	But of course the $U_t$ need not be open. Instead, the $V_t$ will do, for (we~claim)
  	\begin{equation}
  		\text{(a) each $V_t$ is open in $U$,}
  		\qquad
  		\text{(b) their union is dense in $U$.}
  	\end{equation}
	To see (b), note that the closures $\closure(U_t)$ cover $U$: therefore, by \cite[3.11]{Watts:2012} or \cite[7.3]{Choquet:1969}, the interiors $\interior(\closure(U_t))$ have dense union. So it suffices to show that $V_t$ is dense in $\interior(\closure(U_t))$, i.e., any nonempty open $O\subset\interior(\closure(U_t))$ meets $V_t$: this is true because $O\cap V_t$ is just $O\cap U_t$, which is nonempty by openness of $O$ and density of $U_t$ in $\closure(U_t)$. To see (a), note that $U_t$, preimage of the locally closed set $X_t$ by the continuous map $j\circ P$, is locally closed \cite[I.3.3]{Bourbaki:1960}: i.e., $U_t$ is open in $\closure(U_t)$, endowed with its \emph{subspace topology $\sigma$ inside $U$}. It follows that $V_t$ is open in $\interior(\closure(U_t))$, endowed with its \emph{subspace topology $\tau$ inside $\closure(U_t)$}. But $\tau$ is also the subspace topology directly inside $U$ \cite[I.3.1]{Bourbaki:1960}; so (since $\interior(\closure(U_t))$ is euclidean-open) $\tau$-open simply means euclidean-open in $U$; whence (a).
	
	Now we can repeat the argument that led to \eqref{equal_on_subsets} with $V_t$ in place of $U_t$ and get (this time correctly) that $(j\circ P)^*\omega$ and $(j\circ Q)^*\omega$ coincide on the open dense union of the $V_t$, hence everywhere on $U$ by continuity. So \eqref{desired_equal_pull_backs} is proved, and we have our reduced $2$-form.
\end{proof}
	
\begin{proof}[Proof \textup(prequantum case\textup)]
		Let $j$, $\pi$, and $\Phi\inv(0)=\level$ be as in \eqref{prequantum_reduction_diagram}, and for simplicity, rename $\tilde X$ as just $X$ (i.e., drop all tildes in~§\ref{prequantum}). We must prove $P^*j^*\varpi=Q^*j^*\varpi$ for any two plots $P,Q:U\to\level$ with $\pi\circ P=\pi\circ Q$. To this end, let $\bigsqcup_t X_t$ and $\bigsqcup_t\level_t$ be the orbit type partitions of $X$ and $\level$, so that we again have a diagram \eqref{strata}. Following the suggestion in \cite[2.17]{Lerman:2001}, we consider the \emph{symplectization} of $(X,\varpi)$, i.e.~the Hamiltonian $G$-space $(\check X,\check\omega,\check\Phi)$ where $\check X=\RR\times X$ with (proper) $G$-action $g(s,x)=(s,g(x))$, $2$-form $\check\omega = d(\e{s}\varpi)$, and resulting moment map $\check\Phi(s,x) = \e{s}\Phi(x)$. Clearly, its zero level and orbit type pieces are simply the product of $\RR$ with those of $X$, i.e.
	\begin{equation}
		\check\level = \RR\times\level,
		\qquad
		\check X_t = \RR\times X_t,
		\qquad
		\check\level_t = \RR\times\level_t. 
	\end{equation}
	By the symplectic theory expounded before \eqref{reduced_2_forms}, the $\check X_t$ and $\check\level_t$ are $G$-invariant embedded submanifolds of $\check X$, and the reduced pieces $\check\level_t/G$ are symplectic manifolds, carrying $2$-forms $\check\omega_t$ characterized by
	\begin{equation}
		\label{reduced_symplectized_forms}
		(\id\times j_{|\level_t})^*\check\omega =
		(\id\times\pi_{|\level_t})^*\check\omega_t^{\phantom*}.
	\end{equation}
	It follows that $X_t$, $\level_t$ and $\level_t/G$ are manifolds. Moreover, we claim that $\smash{(j_{|\level_t})^*\varpi}$ descends to $\level_t/G$. To see this, note that $\varpi$ can be recovered from $\check\omega$ as
	\begin{equation}
		\label{recovering_varpi}
		\varpi = I^*\i_R\check\omega
	\end{equation}
	where $I:X\to\check X$ is the embedding $x\mapsto(0,x)$, and $\i_R$ is interior product with the constant vector field $R(s,x)=(1,0)$. Then consider the commutative diagram
	\begin{equation}
		\label{four_squares}
	  	\begin{tikzcd}[row sep=large,column sep=large,every label/.append style={font=\scriptsize}]
	            X
				\ar[r,"I"] 
			  & \check X=\RR\times X
			    \ar[r,"R"] 
			  & T\check X
			  \\
			    \level_t
				\ar[r,"I_t"]
				\ar[u,hook,"j_{|\level_t}"]
				\ar[d,swap,"\pi_{|\level_t}"]
			  & \check\level_t=\RR\times\level_t
			    \ar[r,"R_t"]
				\ar[u,hook,swap,"\id\times j_{|\level_t}"]
			    \ar[d,"\id\times\pi_{|\level_t}"]
			  & T\check\level_t
			    \ar[u,hook,swap,"(\id\times j_{|\level_t})_*"]
			    \ar[d,"(\id\times\pi_{|\level_t})_*"]
			  \\
			    \level_t/G
				\ar[r,"J_t"]
			  & \check\level_t/G = \RR\times(\level_t/G)
			    \ar[r,"S_t"]
			  & T(\check\level_t/G)\rlap{,}
	  	\end{tikzcd}
	\end{equation}
	where $I_t$, $J_t$ are again the embeddings $(0,\cdot)$, and $R_t$, $S_t$ the vector fields with constant value $(1,0)$. Note that commutativity of the two squares on the right means that $R$, $R_t$, $S_t$ are related, where we call two vector fields $R:A\to TA$ and $S:B\to TB$ \emph{related} by a smooth map $F:A\to B$ if $F_*(R(a))=S(F(a))$; then one has $\i_RF^*=F^*\i_S$. With that in mind, we can compute
	\begin{equation}
		\label{reduced_1_forms}
		\begin{aligned}
			(j_{|\level_t})^*\varpi
			& = (j_{|\level_t})^*I^*\i_R\check\omega
			&&\text{by \eqref{recovering_varpi}}
			\\
			& = I_t^*(\id\times j_{|\level_t})^*\i_R\check\omega
			&&\text{by commutativity of \eqref{four_squares}}
			\\
			& = I_t^*\i_{R_t}(\id\times j_{|\level_t})^*\check\omega
			&&\text{by relatedness of }R, R_t
			\\
			& = I_t^*\i_{R_t}(\id\times\pi_{|\level_t})^*\check\omega_t^{\phantom*}
			&&\text{by \eqref{reduced_symplectized_forms}}
			\\
			& = I_t^*(\id\times\pi_{|\level_t})^*\i_{S_t}\check\omega_t^{\phantom*}
			&&\text{by relatedness of }R_t, S_t
			\\
			& = (\pi_{|\level_t})^*J_t^*\i_{S_t}\check\omega_t^{\phantom*}
			&&\text{by commutativity of \eqref{four_squares}}
			\\
			& = (\pi_{|\level_t})^*\varpi_t^{\phantom*}
			&&\text{where we define }
			\varpi_t^{\phantom*} = J_t^*\i_{S_t}\check\omega_t^{\phantom*}.
		\end{aligned}
	\end{equation}
	This proves our claim that $(j_{|\level_t})^*\varpi$ is the pull-back of a (unique, ordinary) $1$-form on the manifold $\level_t/G$. Now, using \eqref{reduced_1_forms} just as we used \eqref{reduced_2_forms} earlier, we conclude by the exact same argument that $P^*j^*\varpi = Q^*j^*\varpi$, as desired.
\end{proof}

\begin{coro}\label{sjamaar-lerman-bates}
	In \eqref{reduced_forms_for_proper_actions}\textup, the reduced $2$-form $\omega_{X/\!\!/G}$ on $X/\!\!/G=\level/G$ restricts to the Sjamaar--Lerman--Bates $2$-form $\omega_t$ \eqref{reduced_2_forms} on each reduced piece $\level_t/G$. Likewise\textup, the reduced $1$-form $\varpired$ induces $\varpi_t$ \eqref{reduced_1_forms} on each reduced piece.
\end{coro}

\begin{proof}
	In \eqref{reduced_2_forms}, we have regarded $j_{|\level_t}$ as a map $\level_t\to X$ but, slightly abusively, $\pi_{|\level_t}$ as a map $\level_t\to\level_t/G$ (rather than $\level_t\to\level/G$). With this asymmetry understood and continued, we have the commutative diagram
	\begin{equation}
		\label{induced_forms_on_strata_diagram}
		\begin{tikzcd}[row sep=large,column sep=large,every label/.append style={font=\small}]
			  \level/G
			& \level
			  \ar[l,swap,"\pi"]
			  \ar[r,hook,"j"]
			& X
			\\
			  \level_t/G
			  \ar[u,hook,"i_t"]
			& \level_t
			  \ar[l,swap,"\pi_{|\level_t}"]
			  \ar[u,hook,"j_t"]
			  \ar[ur,swap,"j_{|\level_t}"]
		\end{tikzcd}
	\end{equation}
	where the inductions $i_t$ and $j_t$ are newly named. We must prove $\omega_t^{\phantom*}=i_t^*\omega_{X/\!\!/G}$. We have:
	\begin{equation}
		\label{induced_forms_on_strata_diagram_computation}
		\begin{aligned}
			(\pi_{|\level_t})^*\omega_t
			& = (j_{|\level_t})^*\omega
			&&\text{by \eqref{reduced_2_forms}}
			\\
			& = j_t^*j^*\omega
			&&\text{by commutativity of \eqref{induced_forms_on_strata_diagram}}
			\\
			& = j_t^*\pi^*\omega_{X/\!\!/G}
			&&\text{by definition \eqref{reduced_2-form}}
			\\
			& = (\pi_{|\level_t})^*i_t^*\omega_{X/\!\!/G}
			&&\text{by commutativity of \eqref{induced_forms_on_strata_diagram}}.
		\end{aligned}
	\end{equation}
	As $\pi_{|\level_t}$ is a subduction \eqref{restricted_subduction}, $(\pi_{|\level_t})^*$ is injective \cite[6.39]{Iglesias-Zemmour:2013}, so \eqref{induced_forms_on_strata_diagram_computation} implies the desired equality. The proof that $\varpi_t^{\phantom*}=i_t^*\varpired$ is just the same, with \eqref{reduced_2-form} and \eqref{reduced_2_forms} replaced by \eqref{reduced_1-form} and \eqref{reduced_1_forms}.
\end{proof}

\section{Application: Induction from non-closed subgroups}\label{s:non-closed}

We apply the results of Section~\ref{s:strict} to the quotient of a Lie group $G$ by a non-closed subgroup $H$, the latter endowed with its canonical Lie group structure (\ref{principal}a). Then although $G/H$ is not a manifold (e.g.~if $G$ is a 2-torus and $H$ an irrational winding, we get the famous $\TT_\alpha$), we can still  perfectly well perform the construction (\ref{induced_manifold}/\ref{prequantum_induction}) of $\IND HGY$, for any Hamiltonian or prequantum $H$-space $Y$. The result is not a manifold, and yet:

\begin{coro}
	\label{reduced_form_on_induced_spaces}
	In \textup{(\ref{s-induced}/\ref{p-induced}),} $\IND HGY$ carries a reduced $2$-form \textup(resp.~$1$-form\textup) even when $H$ is not closed.
\end{coro}

\begin{rema}
	\label{parasymplectic_Hamiltonian}
	This 2-form, together with the residual $G$-action and moment map an induced space also carries, make $\IND HGY$ a \emph{parasymplectic Hamiltonian $G$-space} in the sense of \cite{Iglesias-Zemmour:2010,Iglesias-Zemmour:2013,Iglesias-Zemmour:2016}; see \eqref{two_settings}.
\end{rema}

\begin{proof}
	This follows from \eqref{reduced_forms_for_strict_actions}. Indeed: $T^*G$ is a Lie group and $H\subset G\subset T^*G$ is a subgroup, so (\ref{principal}a) says that the $H$-action $h(p)=ph\inv$ on $T^*G$~is principal. By Lemma \eqref{products_and_subsets} below, it follows that the $H$-action $h(p,y)=(ph\inv,h(y))$ on $\psi\inv(0)\subset T^*G\times Y$ (\ref{s-induced}/\ref{p-induced}) is also principal, hence strict; so \eqref{reduced_forms_for_strict_actions} applies.
\end{proof}

\begin{lemm}
	\label{products_and_subsets}
	Let $G$ be a diffeological group acting smoothly on diffeological spaces $X_1$ and $X_2$.
	\begin{compactenum}[\upshape a)]
		\item If the $G$-action on $X_1$ is principal\textup, then so is the diagonal $G$-action on $X=X_1\times X_2$.
		\item If $X_2$ is a $G$-invariant subset of $X_1$ \textup(with subset diffeology\textup{)} and the $G$-action on $X_1$ is principal or strict\textup, then so is the restricted $G$-action on~$X_2$.
	\end{compactenum}
\end{lemm}

\begin{proof}
	(a) First of all, $G$ acting freely on $X_1$ implies the same on $X$. So we need only show that the $G$-action on $X$ is strict. For that (see \ref{smooth_division}), let $P\times Q: U\to X\times X$ be a plot taking values in the graph $\Gamma=\{(x,y):G(x)=G(y)\}$. Since $G$ acts freely on $X$, there is a \emph{unique} (possibly non-smooth) $R:U\to G$ such that $Q(u)=R(u)(P(u))$. Writing $P\times Q=P_1\times P_2\times Q_1\times Q_2$, we see that this is also the unique $R$ such that $Q_1(u)=R(u)(P_1(u))$. Since $G$ acts strictly on $X_1$, this $R$ is \emph{smooth} in an open neighborhood of any $u_0\in U$. So we have \eqref{smooth_division}.
	
	(b) Pick a plot $P_2\times Q_2: U\to X_2\times X_2$ taking values in the graph $\Gamma_2=\{(x_2,y_2):G(x_2)=G(y_2)\}$, and a point $u_0\in U$. Composing $P_2$ and $Q_2$ with the inclusion $X_2\hookrightarrow X_1$, we get a plot $P_1\times Q_1: U\to X_1\times X_1$ taking values in $\Gamma_1=\{(x_1,y_1):G(x_1)=G(y_1)\}$. Since the $G$-action on $X_1$ is strict, there are an open neighborhood $V\ni u_0$, and a plot $R:V\to G$, such that \mbox{$Q_1(u)=R(u)(P_1(u))$} for all $u\in V$. Now this says $Q_2(u)=R(u)(P_2(u))$, i.e.~we have \eqref{smooth_division} for $X_2$. So~(b) is proved for `strict'. For `principal' just add the observation that $G$ acting freely on $X_1$ implies the same on $X_2$.
\end{proof}

\begin{coro}
	The reciprocity theorems \textup{(\ref{symplectic_frobenius}/\ref{prequantum_frobenius})} remain valid when $H$ is not assumed closed.
\end{coro}

\begin{proof}
	Our proofs still work unchanged. (The only difference is that $\IND HGY$ and $M/\!\!/H$ need no longer be manifolds; but they still carry the invoked reduced forms \eqref{reduced_form_on_induced_spaces}, residual $G$-actions, and moment maps \eqref{parasymplectic_Hamiltonian}.)
\end{proof}

\begin{s-induced-from-0}
	\label{s-induced-from-0}
	If $Y$ is the trivial coadjoint orbit~$\{0\}$ of $H$, the induced space \eqref{induced_manifold} boils down to the reduction of $T^*G$ by the cotangent lift of the `right' action of $H$, i.e.
	\begin{equation}
		\label{KMS_space}
		\IND HG\,\{0\} = (T^*G)/\!\!/H.
	\end{equation}
	\emph{If $H$ is closed in $G$} then it is very well known, and indeed the simplest case of Kummer--Marsden--Satzer ``cotangent bundle reduction'' \cite[Thm 2.2.2]{Marsden:2007}, that \eqref{KMS_space} is just $T^*(G/H)$, and under that identification, its reduced $2$-form is just the canonical cotangent bundle $2$-form (so it also inherits a potential $1$-form):
	\begin{equation}
		\label{cotangent_bundle_reduction}
		\omega_{(T^*G)/\!\!/H}=d\varpi_{T^*(G/H)}.
	\end{equation}
	\emph{What if $H$ is not closed}? Faced with a similar situation, the authors of \cite[\nolinebreak 3.3]{Lerman:1993} proposed to take \eqref{cotangent_bundle_reduction} as the \emph{definition} of $T^*(G/H)$ and its $2$-form. (Their reduced space was not diffeological but stratified, for  a proper action.) However, as they immediately observed, this entails the \emph{conjecture} that the result does not depend on the way we write $G/H$ as a quotient \cite[\nolinebreak 3.7]{Lerman:1993}. Here, in contrast, we have the benefit of Iglesias-Zemmour's intrinsic notion of the cotangent space $T^*(X)$ of any diffeological space, with its canonical $2$-form denoted $d\Liouv$ and a Hamiltonian action of $\operatorname{Diff}(X)$ \cite{Iglesias-Zemmour:2010,Iglesias-Zemmour:2013}. So the conjecture becomes the following, which we shall be content to prove when $H$ is \emph{dense}:
\end{s-induced-from-0}

\begin{theo}
	\label{diffeological_KMS}
	Let $G$ be a Lie group and $H$ any dense subgroup. Then the induced space \eqref{KMS_space} with its $2$-form \eqref{reduced_form_on_induced_spaces} and the cotangent space $T^*(G/H)$ of \textup{\cite{Iglesias-Zemmour:2013}} with its $2$-form $d\Liouv$ are isomorphic as parasymplectic Hamiltonian $G$\nobreakdash-spaces.
\end{theo}

\begin{rema}
	The non-dense case presumably boils down to this one by inducing in stages, with the closure $\overline H$ of $H$ as intermediate subgroup. However, the necessary generalization of \cite[2.1]{Ratiu:2022} would take us too far afield.
\end{rema}

\begin{proof}
	For this proof, \emph{identify $\LG^*$ with right-invariant $1$-forms on $G$} by declaring
	\begin{equation}
		\label{right_invariant_form}
		\mu(\delta q) := \<\mu,\delta q.q\inv\>
		\rlap{\qquad\quad$(\mu\in\LG^*, \delta q\in T_qG)$}
	\end{equation}
	(notation \ref{concise}). This satisfies $R_g^*\mu=\mu$ for all $g\in G$, where $R_g(q)=qg$, and any right-invariant $1$-form is so obtained. Next, recall once more that $H$ is canonically a Lie group, with Lie algebra $\LH=\{Z\in\LG:\e{tZ}\in H \text{ for all } t\in\RR\}$ (see \cite[III.4.5]{Bourbaki:1972} or \cite[9.6.13]{Hilgert:2012a}). A key property in our case is that
	\begin{equation}
		\label{G_normalizes_h}
		\text{$G$ normalizes $\LH$:\qquad}
		g\LH g\inv = \LH
		\quad\text{for all }
		g\in G.
	\end{equation}
	Indeed one knows that the normalizer $N_G(\LH)$ is always a closed subgroup containing $H$ \cite[III.9.4, Prop.~10]{Bourbaki:1972}, so it must be $G$ by our density assumption. Now let $X=G/H$, and write $\Pi$ for the projection $G\to G/H$, $\Pi(q)=qH$. 
	
	Step 1 in Iglesias-Zemmour's definition of $T^*(X)$ is to form the space $\Omega^1(X)$ of all diffeological $1$-forms on $X$ \cite[6.28]{Iglesias-Zemmour:2013}. In our case, it is remarkably small and in fact of finite dimension (which we'll see is $\dim(\LG/\LH)$), because we have
	\begin{equation}
		\label{pull_back_to_ann_h}
		\alpha\in\Omega^1(X)
		\qquad\Rightarrow\qquad
		\Pi^*\alpha = \mu
		\quad\text{for some }
		\mu\in\ann(\LH)\subset\LG^*
	\end{equation}
	(notation \ref{cardinal}b, \ref{right_invariant_form}). Indeed, the relation $\Pi\circ R_h=\Pi$ implies $R_h^*\Pi^*\alpha=\Pi^*\alpha$ for all $h\in H$, and since $H$ is dense, the same follows for all $g\in G$: thus $\Pi^*\alpha$ is right-invariant and hence of the form \eqref{right_invariant_form}. To see that $\mu$ must annihilate~$\LH$, note that if $Z\in\LH$ and we put $Q(t)=\e{tZ}$, then $P:=\Pi\circ Q:\RR\to G/H$ is the constant plot $t\mapsto eH$, whence we have $Q^*\mu = P^*\alpha = 0$ \cite[Exerc.~96]{Iglesias-Zemmour:2013}. On the other hand, a straightforward calculation gives $Q^*\mu = \<\mu,Z\>dt$: so we conclude that indeed $\<\mu,Z\>=0$ for all $Z\in\LH$.
	
	Next we claim that, conversely to \eqref{pull_back_to_ann_h}, \emph{every} $\mu\in\ann(\LH)$ is the pull-back of some $\alpha\in\Omega^1(X)$. Indeed, by the same criterion already invoked in \eqref{desired_equal_pull_backs}, this will follow if we show that given any two plots $P:U\to G$ and $Q:U\to G$ with $\Pi\circ P=\Pi\circ Q$, we have $P^*\mu = Q^*\mu$. But $\Pi\circ P=\Pi\circ Q$ means that $R(u):= P(u)\inv Q(u)$ defines a plot $R:U\to H$. Now $(g,gh,h):=(P(u), Q(u), R(u))$ are ordinary smooth functions of $u$, and given $\delta u\in T_uU$ we may compute
	\begin{equation}
		\begin{aligned}
			(Q^*\mu)(\delta u)
			&=\mu(\delta[gh])\\
			&=\mu(\delta g.h + g\delta h)\\
			&=\<\mu,[\delta g.h + g\delta h](gh)\inv\>
			&&\text{by \eqref{right_invariant_form}}\\
			&=\<\mu,\delta g.g\inv + g\delta h.h\inv g\inv\>\\
			&=\<\mu,\delta g.g\inv\>
			&&\text{by (\ref{G_normalizes_h}, \ref{pull_back_to_ann_h})}\\
			&=(P^*\mu)(\delta u).
		\end{aligned}
	\end{equation}
	This establishes our claim that $\Pi^*$ is onto $\ann(\LH)$. As $\Pi$ is a subduction, $\Pi^*$ is also one-to-one \cite[6.39]{Iglesias-Zemmour:2013}, hence we have a (linear) bijection $\Omega^1(X)\to\ann(\LH)$. Writing $\Pi_*$ for its inverse, we summarize the argument so far in the following statement (which has since been generalized in \cite{Clark:2024a}):
	\begin{prop}
		\label{Pi_*}
		If $H$ is dense in $G$\textup, then $\Omega^1(G/H) = \Pi_*(\ann(\LH))$.\qed
	\end{prop}
	
	\begin{proof}[Proof of \eqref{diffeological_KMS}\,\textup(continued\textup)] In Step 2 of his construction, \cite[6.40]{Iglesias-Zemmour:2013} defines the cotangent space at $x\in X$ as the quotient (also denoted $\Lambda^1_x(X)$) 
	\begin{equation}
		\label{cotangent_space_at_x}
		T^*_x(X) = \Omega^1(X)/\{1\textrm{-forms vanishing at }x\},
	\end{equation}
	where $\alpha\in\Omega^1(X)$ is said to \emph{vanish at} $x$ if, whenever $P:U\to X$ is a plot with $0\in U$ and $P(0) = x$, the ordinary $1$-form $P^*\alpha\in\Omega^1(U)$ vanishes at $0$. Now in our case \eqref{pull_back_to_ann_h}, we claim: \emph{$\alpha = \Pi_*(\mu)$ vanishes nowhere unless $\mu=0$}. Indeed, given $x=qH\in X$ and $Z\in\LG$, put $Q(t)=\e{tZ}q$. Then one checks readily that the plot $P=\Pi\circ Q:\RR\to X$ has $P(0)=x$ and $P^*\alpha = \<\mu,Z\>dt$. The only way this can vanish at $t=0$ for all $Z\in\LG$ is if $\mu=0$, so we have our claim.
	
	Step 3 of \cite[6.45,~6.48]{Iglesias-Zemmour:2013} defines $T^*(X)$ (also denoted $\mathbf{\Lambda}^1(X)$) as $\{(x,a): x\in X, a\in T^*_x(X)\}$. For us \eqref{cotangent_space_at_x} has no denominator, so this is simply
	\begin{equation}
		\label{cotangent_space}
		\begin{aligned}
			T^*(X) &= X\times\Omega^1(X)\\
			&= (G/H)\times\Pi_*(\ann(\LH))\rlap{\qquad by \eqref{Pi_*}.}
		\end{aligned}
	\end{equation}
	
	Step 4 of \cite[6.49]{Iglesias-Zemmour:2013} next defines the \emph{Liouville $1$-form} on $T^*(X)$ as follows. Recall that a differential form on a diffeological space is defined by specifying its pull-back by every plot of the space. Now the plots of \eqref{cotangent_space} have the form $P\times A$ for plots $P:U\to X$ and $A:U\to\Omega^1(X)$; the definition, then, is:
	\begin{equation}
		\label{Liouv}
		((P\times A)^*\Liouv)(\delta u)
		= P^*(A(u))(\delta u)
		\rlap{\qquad($\delta u\in T_uU$).}
	\end{equation}
	(In general this defines a $1$-form `Taut' on $X\times\Omega^1(X)$, which then descends as Liouv on the quotient $T^*(X)$. For us there is no division, so Liouv $=$~Taut.)
	
	Step 5 of \cite[6.52]{Iglesias-Zemmour:2013} defines a $\Liouv$-preserving action of $G$ (or $\operatorname{Diff}(X)$) on \eqref{cotangent_space} by $g(x,\alpha)=(g(x),g_*(\alpha))$. As our action is by left translations and $\Pi^*\alpha$ is right-invariant, we may as well push it forward by $g(\,\cdot\,)g\inv$; so we get
	\begin{equation}
		\label{lifted_action}
		y=(\Pi(q),\Pi_*(\mu))
		\quad\Rightarrow\quad
		g(y) = (\Pi(gq),\Pi_*(g\mu g\inv)).
	\end{equation}
	
	Finally Step 6 of \cite[9.11]{Iglesias-Zemmour:2013} defines the \emph{moment map} $\Phi:T^*(X)\to\LG^*$ by the following formula, where $y\in T^*(X)$ and 
	$\hat y:G\to T^*(X)$ is the map $g\mapsto g(y)$:
	\begin{equation}
		\label{cotangent_moment}
		\Phi(y) =
		\text{value at $e$ of the $1$-form }
		\hat y^*\Liouv\in\Omega^1(G).
	\end{equation}
	
	It is now easy to compare (\ref{cotangent_space}--\ref{cotangent_moment}) with (\ref{KMS_space}, \ref{reduced_form_on_induced_spaces}--\ref{parasymplectic_Hamiltonian}). Indeed, a glance at \eqref{phi_and_psi} reveals that $\IND HG\,\{0\}$ is the quotient of
	\begin{equation}
		\label{level}
		\begin{aligned}
			\psi\inv(0)
			&=\{p=\mu q: q\in G, \mu\in q\ann(\LH)q\inv\}\\
			&=\{p=\mu q: q\in G, \mu\in \ann(\LH)\}
			\rlap{\qquad\text{by \eqref{G_normalizes_h}}}\\
			&\cong \,G\times\ann(\LH)
		\end{aligned}
	\end{equation}
	(notation \ref{concise}) by the $H$-action, where we note that under this last identification, the $G\times H$-action \eqref{s-induced} and the injection $j:\psi\inv(0)\hookrightarrow T^*G$ become
	\begin{equation}
		\label{trivialization}
		(g,h)(q,\mu)=(gqh\inv,g\mu g\inv)
		\qquad\text{and}\qquad
		j(q,\mu)=\mu q.
	\end{equation}
	Now clearly $F(q,\mu) := (\Pi(q),\Pi_*(\mu))$ defines a map from \eqref{level} to \eqref{cotangent_space} which is $G$-equivariant \eqref{lifted_action} and descends to a diffeomorphism $\psi\inv(0)/H\to T^*(G/H)$. There remains to see that
	\begin{equation}
		\label{pull_backs}
	   \text{(a) }\ F^*\Liouv = j^*\varpi_{T^*G}
	   \qquad\text{and}\qquad
	   \text{(b) }\ \Phi\circ F=\phi\circ j
	\end{equation}
	with $\phi$ as in \eqref{phi_and_psi}. Now (a) means that $(Q\times M)^*F^*\Liouv = (Q\times M)^*j^*\varpi_{T^*G}$ for all plots $Q\times M:U\to G\times\ann(\LH)$. To prove this, note that the left-hand side is $\Liouv$'s pull-back by $F\circ(Q\times M) = (\Pi\circ Q)\times(\Pi_*\circ M)$, which we can take as $P\times A$ in \eqref{Liouv}. Thus we obtain
	\begin{equation}
		\begin{aligned}
			((Q\times M)^*F^*\Liouv)(\delta u)
			&= (\Pi\circ Q)^*((\Pi_*\circ M)(u))(\delta u)
			&&\text{by \eqref{Liouv}}\\
			&= Q^*(\Pi^*(\Pi_*(M(u))))(\delta u)\\
			&= Q^*(M(u))(\delta u)\\
			&= M(u)(\delta[Q(u)])\\
			&= \<M(u),\delta[Q(u)]Q(u)\inv\>
			&&\text{by \eqref{right_invariant_form}}\\
			&= \<M(u)Q(u),\delta[Q(u)]\>
			&&\text{by \eqref{concise}}\\
			&=\varpi_{T^*G}(\delta[M(u)Q(u)])\\
			&=\varpi_{T^*G}(\delta[j(Q(u),M(u))])
			&&\text{by \eqref{trivialization}}\\
			&= ((Q\times M)^*j^*\varpi_{T^*G})(\delta u)
		\end{aligned}
	\end{equation}
	as desired. To prove (\ref{pull_backs}b), note that $(\phi\circ j)(q,\mu)=\phi(\mu q)=\mu$ \eqref{phi_and_psi}, whereas if $y=F(q,\mu)$, then $G$-equivariance gives $\hat y = F\circ(q,\mu)\hat{\phantom.}$ where $(q,\mu)\hat{\phantom.}(g)=g(q,\mu)$. So \eqref{cotangent_moment} says that $(\Phi\circ F)(q,\mu)$ is the value at $e$ of the (\emph{left}-invariant) 1-form
	\begin{equation}
		\begin{aligned}
			(\hat y^*\Liouv)(\delta g)
			&=((q,\mu)\hat{\phantom.}^*F^*\Liouv)(\delta g)\\
			&=((q,\mu)\hat{\phantom.}^*j^*\varpi_{T^*G})(\delta g)
			&&\text{by (\ref{pull_backs}a)}\\
			&=\varpi_{T^*G}(\delta[j(g(q,\mu))])\\
			&=\varpi_{T^*G}(\delta[g\mu q])
			&&\text{by \eqref{trivialization}}\\
			&=\<g\mu q,\delta g.q\>\\
			&=\<\mu,g\inv\delta g\>
			&&\text{by \eqref{concise},}
		\end{aligned}
	\end{equation}
	i.e.~$(\Phi\circ F)(q,\mu)=\mu$. So both sides of (\ref{pull_backs}b) are equal, as claimed.
\end{proof}\let\qed\relax
\end{proof}

\bookmarksetup{startatroot}
\section*{Acknowledgements}
The first named author would like to thank Mauro Spera for invaluable guidance and help in the initial stages of this work. His research is carried out within INDAM-GNSAGA's framework.

\let\i\dotlessi
\let\l\polishl
\let\o\norwegiano
\let\u\russianbreve

\bibliographystyle{amsalpha}

\end{document}